\documentclass[11pt, twoside]{article}
\usepackage{latexsym}
\usepackage{amsmath}
\usepackage{amssymb}
\usepackage[all]{xy}
\usepackage{amsfonts}
\usepackage{verbatim}
\usepackage{amsthm}
\usepackage{mathrsfs}
\usepackage{epsfig}
\usepackage{xy}
\usepackage{tensor}
\usepackage{array}
\usepackage{stmaryrd}
\usepackage{graphicx,color}
\usepackage{xcolor}
\usepackage[colorlinks=true,linkcolor=blue,citecolor=blue]{hyperref}
\usepackage{tikz}
\usetikzlibrary{arrows,calc}
\usepackage{etex}
\usepackage{mathdots}
\usepackage{float}
\usepackage{graphics}
\usepackage{pdflscape}
\usepackage{extarrows}
\usepackage{anysize,hyperref}
\input xypic
\xyoption{all}
\usepackage{bm}
\usepackage[perpage,symbol]{footmisc}
\usepackage{setspace}
\def\P{\mathcal{P}}
\def\I{\mathcal{I}}

\def\C{\mathscr{C}}
\def\E{\mathbb{E}}

\def\s{\mathfrak{s}}

\def\op{^\mathrm{op}}
\def\Ab{\mathsf{Ab}}
\def\del{\delta}
\def\dr{\ar@{->}[r]}

\def\X{\mathscr{X}}

\def\Y{\mathscr{Y}}

\newcommand{\CC}{{\bf{C}}^{n+2}_{\C}}
\newcommand{\mr}{\hbox{\boldmath$\cdot$}}
\newcommand{\ov}{\overset}
\newcommand{\lra}{\longrightarrow}
\newcommand{\co}{\colon}
\newcommand{\uas}{^{\ast}}            
\newcommand{\sas}{_{\ast}}
\newcommand{\Xd}{\langle X^{\mr},\del\rangle}  
\newcommand{\Yr}{\langle Y^{\mr},\rho\rangle}  

\newcommand{\ush}{^\sharp}           
\newcommand{\ssh}{_\sharp}
\SelectTips{cm}{10}

\usepackage{enumitem}

\begin{document}
\baselineskip=15pt
\title{\Large{\bf Higher Auslander's defect and classifying substructures of $\bm{n}$-exangulated categories \footnotetext{\hspace{-1em}$^\ast$Corresponding author. Jiangsheng Hu was supported by the NSF of China (Grant Nos. 12171206 and 11771212) and the Natural Science Foundation of Jiangsu Province (Grant No. BK20211358). Panyue Zhou was supported by the National Natural Science Foundation of China (Grant No. 11901190) and the Scientific Research Fund of Hunan Provincial Education Department (Grant No. 19B239).}}}
\medskip
\author{Jiangsheng Hu, Yajun Ma, Dongdong Zhang$^\ast$ and Panyue Zhou}

\date{}

\maketitle
\def\blue{\color{blue}}
\def\red{\color{red}}

\newtheorem{theorem}{Theorem}[section]
\newtheorem{lemma}[theorem]{Lemma}
\newtheorem{corollary}[theorem]{Corollary}
\newtheorem{proposition}[theorem]{Proposition}
\newtheorem{conjecture}{Conjecture}
\theoremstyle{definition}
\newtheorem{definition}[theorem]{Definition}
\newtheorem{question}[theorem]{Question}
\newtheorem{remark}[theorem]{Remark}
\newtheorem{remark*}[]{Remark}
\newtheorem{example}[theorem]{Example}
\newtheorem{example*}[]{Example}
\newtheorem{condition}[theorem]{Condition}
\newtheorem{condition*}[]{Condition}
\newtheorem{construction}[theorem]{Construction}
\newtheorem{construction*}[]{Construction}

\newtheorem{assumption}[theorem]{Assumption}
\newtheorem{assumption*}[]{Assumption}

\baselineskip=17pt
\parindent=0.5cm

\begin{abstract}
\baselineskip=16pt
Herschend-Liu-Nakaoka introduced the notion of $n$-exangulated
categories. It is not only a higher
dimensional analogue of extriangulated categories defined by Nakaoka-Palu,
but also gives a simultaneous generalization of $n$-exact categories and $(n+2)$-angulated categories.  In this article, we give an $n$-exangulated version of Auslander's defect and  Auslander-Reiten duality formula. Moreover, we also give a classification of substructures (=closed subbifunctors) of a given skeletally small $n$-exangulated category by using the category of defects. \\[0.3cm]
\textbf{Keywords:} $n$-exangulated categories; defects;
 Auslander's defect formula; substructures. \\[0.1cm]
\textbf{ 2020 Mathematics Subject Classification:} 18G80; 18E10; 18E05.
\medskip
\end{abstract}

\pagestyle{myheadings}
\markboth{\rightline {\scriptsize J. Hu, Y. Ma, D. Zhang, P. Zhou\hspace{2mm}}}
         {\leftline{\scriptsize Higher Auslander's defect and classifying substructures of $n$-exangulated categories}}

\section{Introduction}
The notion of extriangulated categories was introduced in \cite{NP}, which can be viewed
 as a simultaneous generalization of exact categories and triangulated categories.
 The data of such a category is a triplet $(\C,\E,\s)$, where $\C$ is an additive category, $\mathbb{E}: \C^{\rm op}\times \C \rightarrow \Ab$ is an additive bifunctor and $\mathfrak{s}$ assigns to each $\delta\in \mathbb{E}(C,A)$ a class of $3$-term sequences with end terms $A$ and $C$ such that certain axioms hold. Recently, Herschend-Liu-Nakaoka \cite{HLN} introduced an $n$-analogue of this notion called $n$-exangulated category. Such a category is a similar triplet $(\C,\E,\s)$, with the main distinction being that the $3$-term sequences mentioned above are replaced by $(n+2)$-term sequences. It should be noted that the case $n =1$ corresponds to extriangulated categories. As typical examples we have that $n$-exact and $(n+2)$-angulated categories are $n$-exangulated, see \cite[Propositions 4.34 and 4.5]{HLN}. However,
 there are some other examples of $n$-exangulated categories which are neither $n$-exact nor $(n+2)$-angulated, see \cite[Remark 4.39]{HLN} and \cite[Remark 4.5]{LZ}.

The notion of defects over abelian categories was introduced by Auslander in \cite{A1}, see also
\cite[IV. 4]{ARS}. Recently, this notion was generalized by Jasso-Kvamme \cite{JK} over $n$-exact categories and by Ogawa \cite{Oga} over extriangulated categories. In this paper, we define defects over $n$-exangulated categories. More precisely,
let $(\C,\E,\s)$ be an $n$-exangulated category.
If
  $$X^{0}\xrightarrow{d_X^{0}} X^{1}\xrightarrow{d_X^{1}} \cdots\xrightarrow{} X^{n-1}\xrightarrow{d_X^{n-1}} X^{n}\xrightarrow{d_X^n} X^{n+1}\stackrel{\delta}{\dashrightarrow},$$
is a distinguished $n$-exangle in $\C$, then the contravariant defect of $\delta$, denoted by $\widetilde{\delta}$, is defined by
$$\C(-, X^{n})\xrightarrow{\C(-,\ d_X^n)} \C(-, X^{n+1})\rightarrow \widetilde{\delta}\rightarrow 0 $$ the exact sequence of the functors. The covariant defect of $\delta$, denoted by $\check{\delta}$, is defined dually.

 One of interesting applications of defects, due to Jasso-Kvamme \cite[Theorem 3.8]{JK}, is that one can get higher version of Auslander's defect formula using a minor modification of Krause's proof of the classical formula \cite{K}. As an immediate consequence, one can obtain the higher Auslander-Reiten duality formula; see \cite[Corollary 3.9]{JK} for instance. Recall that Auslander's defect formula appeared as \cite[Theorem III.4.1]{A2} for the first time and the Auslander-Reiten duality formula appeared at about the same time as Proposition 3.1 in Auslander and Reiten's classical paper on the existence of almost split sequences \cite{AR2}. Recently, some analogues of the above formulas can be seen in \cite{I,J,L1,L2}. Our first main result is the following, which provides a higher counterpart of Auslander's defect and Auslander-Reiten duality formula over $n$-exangulated categories.

\begin{theorem}{\rm (see Theorem \ref{thm4} and Corollary \ref{corollary:Auslander-Reiten duality} for details)} \label{thm:1.2}
Let $(\C, \mathbb{E}, \mathfrak{s})$ be an $n$-exangulated category with enough projectives and enough injectives. Assume that $\C$ is a dualizing $k$-variety. Then there exists an equivalence $\tau\colon\underline{\C}\rightarrow \overline{\C}$ satisfying the following properties:
\begin{itemize}
\item[\rm (1)]  $D\check{\delta}=\widetilde{\delta}\tau^{-1},$ $D\widetilde{\delta}=\check{\delta}\tau$;

\item[\rm (2)] $D\E(-,X)\cong\underline{\C}(\tau^{-1}X,-),$ $D\E(X,-)\cong\overline{\C}(-,\tau X),$
\end{itemize}
where $\underline{\C}$ (resp.   $\overline{\C}$) is the projectively (resp. injectively) stable category of $\C$ and $D$ is the $k$-dual.
 \end{theorem}

There are two remarks on Theorem \ref{thm:1.2}: (1) It generalizes Jasso-Kvamme's results on
the category of finitely generated left $\Lambda$-modules over an Artin algebra $\Lambda$ and Lin's results on an exact category and is new for $n$-cluster tilting subcategories of an extriangulated category (see Corollary \ref{corollary:3.15}); (2)  As a consequence of Theorem \ref{thm:1.2}, we give a method to decide whether an $(n+2)$-angulated category has Auslander-Reiten $(n+2)$-angles (see Corollary \ref{corollary:3.14}).

Another interesting application of defects, due to Enomoto \cite{Eno}, is that one can give a classification of substructures of a given skeletally small extriangulated category. We show that this result has a higher counterpart:

\begin{theorem}{\rm (see Theorem \ref{main47} for details)}
Let  $(\C, \mathbb{E}, \mathfrak{s})$ be a skeletally small $n$-exangulated category with weak-kernel. Denote by ${\rm def}\mathbb{E}$ the full subcategory of finitely presented functors isomorphic to contravariant defects.
Then the map $\mathbb{F}\mapsto {\rm def}\mathbb{F}$ gives an isomorphism of the following posets, where the poset structures are given by inclusion.
\begin{itemize}
\item[{\rm (1)}] The poset of closed subbifunctors of $\mathbb{E}$;

\item[{\rm (2)}] The poset of  Serre subcategories of ${\rm def}\mathbb{E}$.
\end{itemize}
\end{theorem}

Note that, in \cite{Eno}, one of the key arguments in the proof is that any  $\mathbb{E}$-triangle $$\xymatrix{A\ar[r]^x&B\ar[r]^{y}&C\ar@{-->}[r]^{\delta}&}$$ gives rise to a weak pullback for any morphism $c:C'\rightarrow C$ (see \cite[Lemma 2.5]{Eno}) and its proof relies on the fact that any extriangulated category has shifted octahedrons, while in our general context we do not have this fact and therefore must avoid this kind of arguments.

This article is organized as follows. In Section 2, we review some elementary definitions and facts on $n$-exangulated categories.
In Section 3, we prove an $n$-exangulated version of Auslander's defect formula and Auslander-Reiten duality formula (see Theorem \ref{thm4}). As an application, we give a method to decide whether an $(n+2)$-angulated category has Auslander-Reiten $(n+2)$-angles. In Section 4, we give the proof of Theorem \ref{main47}, and some applications are given.

\section{Preliminaries}\label{section2}
In this section, we first review some definitions and basic facts.

\noindent{\bf 2.1 Morphism categories.} Let $\C$ be an additive category. The {\it morphism category} of $\C$ is the category ${\rm Mor}\C$ defined by the following data. The objects of ${\rm Mor}\C$ are all the morphisms $f:X\rightarrow Y$ in $\C$. The morphism from $f:X\rightarrow Y$ to $f': X'\rightarrow Y'$ are pairs $(a,b)$ where $a: X\rightarrow X'$ and $b: Y\rightarrow Y'$ satisfies $bf=f'a$. The composition of morphisms is componentwise.

\noindent{\bf 2.2 Functor categories.} For an additive category $\C$, a {\it right $\C$-module} is defined to be a contravariant functor $\C^{op}\rightarrow \Ab$ and a morphism between right $\C$-modules is just a natural transformation of functors. Thus we define an abelian category ${\rm Mod}\C$. We denote by ${\rm mod}\C$ the full subcategory of finitely presented $\C$-module in ${\rm Mod}\C$ and by $\mathrm{proj}\text{-}\C$ (resp. $\mathrm{inj}\text{-}\C$) the full subcategory of ${\rm mod}\C$ consisting of
projective (resp. injective) objects. For an additive category $\C$, the subcategory ${\rm mod}\C$ is not always abelian since it is not necessarily closed under kernels. Indeed, we have following well-known result.

\begin{lemma} Let $\C$ be an additive category. Then the full subcategory ${\rm mod}\C$ is an exact abelian subcategory in ${\rm Mod}\C$ if and only if the  category $\C$ admits weak-kernels.
\end{lemma}

\noindent{\bf 2.3 Quotient categories.} Let $\C$ be an additive category and $\mathcal{J}$ an ideal of $\C$. The {\it quotient category}  $\C/\mathcal{J}$ has the same objects as $\C$, and the morphisms in $\C/\mathcal{J}$ are the residue classes of homomorphisms in $\C$ modulo $\mathcal{J}$. Suppose that $\X$ is a full subcategory of $\C$, we denote by $[\X]$ the ideal of $\C$ given by all morphisms factors through an object in $\X$, thus we have a quotient category $\C/[\X]$.

Let $F: \C\rightarrow \mathcal{D}$ be an additive functor. An object $A$ of $\C$ is called a {\it kernel object} for $F$ provided $F(A)=0$. The functor $F: \C\rightarrow \mathcal{D}$ will be said to be {\it objective} provided any morphism $f: A\rightarrow A'$ with $F(f)=0$ factors through a kernel object for $F$, see \cite[Appendix]{RZ}. If $F: \C\rightarrow \mathcal{D}$ is a full, dense, objective functor, there is an equivalence $\C/[{\rm Ker}F]\cong \mathcal{D}$, where ${\rm Ker}F$ is the full subcategory of $\C$ formed by $X$ with $F(X)=0$.

\noindent{\bf 2.4 $n$-exangulated categories.} Let $n$ be a positive integer and $\C$  an additive category equipped with an additive bifunctor $\E\colon\C\op\times\C\to {\Ab}$, where ${\Ab}$ is the category of abelian groups. Next we briefly recall some definitions and basic properties of $n$-exangulated categories from \cite{HLN}.

{ For any pair of objects $A,C\in\C$, an element $\del\in\E(C,A)$ is called an {\it $\E$-extension} or simply an {\it extension}. We also write such $\del$ as $\tensor[_A]{\del}{_C}$ when we indicate $A$ and $C$. The zero element $\tensor[_A]{0}{_C}=0\in\E(C,A)$ is called the {\it split $\E$-extension}. For any pair of $\E$-extensions $\tensor[_A]{\del}{_C}$ and $\tensor[_{A'}]{{{\delta}{'}}}{_{C'}}$, let $\delta\oplus \delta'\in\mathbb{E}(C\oplus C', A\oplus A')$ be the
element corresponding to $(\delta,0,0,{\delta}{'})$ through the natural isomorphism $\mathbb{E}(C\oplus C', A\oplus A')\simeq\mathbb{E}(C, A)\oplus\mathbb{E}(C, A')
\oplus\mathbb{E}(C', A)\oplus\mathbb{E}(C', A')$.

For any $a\in\C(A,A')$ and $c\in\C(C',C)$,  $\E(C,a)(\del)\in\E(C,A')\ \ \text{and}\ \ \E(c,A)(\del)\in\E(C',A)$ are simply denoted by $a_{\ast}\del$ and $c^{\ast}\del$, respectively.

Let $\tensor[_A]{\del}{_C}$ and $\tensor[_{A'}]{{{\delta}{'}}}{_{C'}}$ be any pair of $\E$-extensions. A {\it morphism} $(a,c)\colon\del\to{\delta}{'}$ of extensions is a pair of morphisms $a\in\C(A,A')$ and $c\in\C(C,C')$ in $\C$, satisfying the equality
$a_{\ast}\del=c^{\ast}{\delta}{'}$.}

\begin{definition}\cite[Definition 2.11]{HLN}
By  the Yoneda lemma, any extension $\del\in\E(C,A)$ induces natural transformations
\[ \del\ssh\colon\C(-,C)\Rightarrow\E(-,A)\ \ \text{and}\ \ \del\ush\colon\C(A,-)\Rightarrow\E(C,-). \]
For any $X\in\C$, these $(\del\ssh)_X$ and $\del\ush_X$ are given as follows.
\begin{enumerate}
\item[\rm(1)] $(\del\ssh)_X\colon\C(X,C)\to\E(X,A)\ ;\ f\mapsto f\uas\del$.
\item[\rm (2)] $\del\ush_X\colon\C(A,X)\to\E(C,X)\ ;\ g\mapsto g\sas\delta$.
\end{enumerate}
\end{definition}

\begin{definition}\cite[Definition 2.7]{HLN}
Let $\bf{C}_{\C}$ be the category of complexes in $\C$. As its full subcategory, define $\CC$ to be the category of complexes in $\C$ whose components are zero in the degrees outside of $\{0,1,\ldots,n+1\}$. Namely, an object in $\CC$ is a complex $X^{\mr}=\{X^i,d_X^i\}$ of the form
\[ X^0\xrightarrow{d_X^0}X^1\xrightarrow{d_X^1}\cdots\xrightarrow{d_X^{n-1}}X^n\xrightarrow{d_X^n}X^{n+1}. \]
We write a morphism $f^{\mr}\co X^{\mr}\to Y^{\mr}$ { simply as} $f^{\mr}=(f^0,f^1,\ldots,f^{n+1})$, only indicating the terms of degrees $0,\ldots,n+1$.

 We define the homotopy relation on the morphism sets in the usual way. Denote by $\mathbf{K}_{\C}^{n+2}$ the homotopy category, which is the quotient of $\CC$ by the ideal of null-homotopic morphisms.
\end{definition}

\begin{definition}\cite[Definitions 2.9 and 2.13]{HLN}
 Let $\C,\E,n$ be as before. Define a category $\AE:=\AE^{n+2}_{(\C,\E)}$ as follows.
\begin{enumerate}
\item[\rm(1)]  A pair $\Xd$ is an object of the category $\AE$ with $X^{\mr}\in\CC$
and $\del\in\E(X^{n+1},X^0)$, called an {\it $\E$-attached
complex of length} $n+2$, if it satisfies
$$(d^0_X)_{\ast}\del=0~~\textrm{and}~~(d_X^n)^{\ast}\del=0.$$
We also denote it by
$$X^0\xrightarrow{d^0_X}X^1\xrightarrow{d^1_X}\cdots\xrightarrow{d^{n-2}_X}X^{n-1}
\xrightarrow{d^{n-1}_X}X^n\xrightarrow{d^n_X}X^{n+1}\overset{\delta}{\dashrightarrow}$$
\item[\rm (2)]  For such pairs $\Xd$ and $\langle Y^{\mr},\rho\rangle$, a morphism  $f^{\mr}\colon\Xd\to\langle Y^{\mr},\rho\rangle$ in $\AE$ is
defined to be a morphism in $\CC$ satisfying $(f^0)_{\ast}\del=(f^{n+1})^{\ast}\rho$.
{ We use the same composition and the identities as in $\CC$.}
\end{enumerate}

An {\it $n$-exangle} is a pair  $\Xd$ of $X^{\mr}\in\CC$ and $\delta\in\mathbb{E}(X^{n+1},X^0)$ which satisfies the following conditions.
\begin{enumerate}
\item[\rm (a)] The following sequence of functors $\C\op\to\Ab$ is exact.
$$
\C(-,X^0)\xLongrightarrow{\C(-,\ d_X^0)}\cdots\xLongrightarrow{\C(-,\ d_X^n)}\C(-,X^{n+1})\xLongrightarrow{~\del\ssh~}\E(-,X^0)
$$
\item[\rm (b)] The following sequence of functors $\C\to\Ab$ is exact.
$$
\C(X^{n+1},-)\xLongrightarrow{\C(d_X^n,\ -)}\cdots\xLongrightarrow{\C(d_X^0,\ -)}\C(X^0,-)\xLongrightarrow{~\del\ush~}\E(X^{n+1},-)
$$
\end{enumerate}
In particular any $n$-exangle is an object in $\AE$.
A {\it morphism of $n$-exangles} simply means a morphism in $\AE$. Thus $n$-exangles form a full subcategory of $\AE$.
\end{definition}

{ \begin{definition}\cite[Definition 2.17]{HLN}
Let $A,C\in{\C}$ be any pair of objects. The subcategory of $\CC$, denoted by $\mathbf{C}^{n+2}_{(\C;A,C)}$, or simply
by $\mathbf{C}^{n+2}_{(A,C)}$, is defined as follows.

\begin{enumerate}
\item[\rm(1)] An object $X^{\mr}\in{\CC}$ is in $\mathbf{C}^{n+2}_{(A,C)}$ if it satisfies $X^0=A$ and $X^{n+1}=C$. We also write it as $_{A}X^{\mr}_{C}$ when we emphasize $A$ and $C$.
\item[\rm(2)] For any $X^{\mr}, Y^{\mr}\in{\mathbf{C}^{n+2}_{(A,C)}}$, the morphism set is defined by
    $$\mathbf{C}^{n+2}_{(A,C)}(X^{\mr},Y^{\mr})=\{f^{\mr}\in{\CC(X^{\mr},Y^{\mr})}\ | \ f^{0}=1_{A}, \ f^{n+1}=1_{C}\}.$$
\end{enumerate}

We denote by $\mathbf{K}_{(A,C)}^{n+2}$ the quotient of $\mathbf{C}^{n+2}_{(A,C)}$
 by the same homotopy relation as $\mathbf{C}^{n+2}_{\C}$. The homotopy equivalence class of
$_{A}X^{\mr}_{C}$ is denoted by $[_{A}X^{\mr}_{C}]$, or simply by $[X^{\mr}]$.
\end{definition}}

{ \begin{remark}\cite[Remark 2.18]{HLN}
 Let $X^{\mr}, Y^{\mr}\in{\mathbf{C}^{n+2}_{(A,C)}}$ be any pair of objects. If a morphism $f^{\mr}\in\mathbf{C}^{n+2}_{(A,C)}(X^{\mr},Y^{\mr})$ gives a homotopy equivalence in $\mathbf{C}^{n+2}_{\C}$, then it is a homotopy equivalence in $\mathbf{C}^{n+2}_{(A,C)}$. However, the converse is not true in general. Thus there can be a difference between homotopy equivalences taken in ${\mathbf{C}^{n+2}_{(A,C)}}$ and $\mathbf{C}^{n+2}_{\C}$.
\end{remark}}

\begin{definition}\cite[Definition 2.22]{HLN}\label{def1}
Let $\s$ be a correspondence which associates a homotopy equivalence class $\s(\del)=[{}_AX^{\mr}_C]$ to each extension $\del={ \tensor[_A]{\delta}{_C}}$. Such { an} $\s$ is called a {\it realization} of $\E$ if it satisfies the following condition for any $\s(\del)=[X^{\mr}]$ and any $\s(\rho)=[Y^{\mr}]$.
\begin{itemize}
\item[{\rm (R0)}] For any morphism of extensions $(a,c)\co\del\to\rho$, there exists a morphism $f^{\mr}\in\CC(X^{\mr},Y^{\mr})$ of the form $f^{\mr}=(a,f^1,\ldots,f^n,c)$. Such { an} $f^{\mr}$ is called a {\it lift} of $(a,c)$.
\end{itemize}
In such a case, we { simply} say that \lq\lq$X^{\mr}$ realizes $\del$" whenever they satisfy $\s(\del)=[X^{\mr}]$.

Moreover, a realization $\s$ of $\E$ is said to be {\it exact} if it satisfies the following conditions.
\begin{itemize}
\item[{\rm (R1)}] For any $\s(\del)=[X^{\mr}]$, the pair $\Xd$ is an $n$-exangle.
\item[{\rm (R2)}] For any $A\in\C$, the zero element ${ \tensor[_A]{0}{_0}}=0\in\E(0,A)$ satisfies
\[ \s({ \tensor[_A]{0}{_0}})=[A\ov{1_A}{\lra}A\to0\to\cdots\to0\to0]. \]
Dually, $\s({ \tensor[_0]{0}{_A}})=[0\to0\to\cdots\to0\to A\ov{1_A}{\lra}A]$ holds for any $A\in\C$.
\end{itemize}
Note that the above condition {\rm (R1)} does not depend on representatives of the class $[X^{\mr}]$.
\end{definition}

\begin{definition}\cite[Definition 2.23]{HLN}
Let $\s$ be an exact realization of $\E$.
\begin{enumerate}
\item[\rm (1)] An $n$-exangle $\Xd$ is called an $\s$-{\it distinguished} $n$-exangle if it satisfies $\s(\del)=[X^{\mr}]$. We often simply say {\it distinguished $n$-exangle} when $\s$ is clear from the context.
\item[\rm (2)]  An object $X^{\mr}\in\CC$ is called an {\it $\s$-conflation} or simply a {\it conflation} if it realizes some extension $\del\in\E(X^{n+1},X^0)$.
\item[\rm (3)]  A morphism $f$ in $\C$ is called an {\it $\s$-inflation} or simply an {\it inflation} if it admits some conflation $X^{\mr}\in\CC$ satisfying $d_X^0=f$.
\item[\rm (4)]  A morphism $g$ in $\C$ is called an {\it $\s$-deflation} or simply a {\it deflation} if it admits some conflation $X^{\mr}\in\CC$ satisfying $d_X^n=g$.
\end{enumerate}
\end{definition}

\begin{definition}\cite[Definition 2.27]{HLN}
For a morphism $f^{\mr}\in\CC(X^{\mr},Y^{\mr})$ satisfying $f^0=1_A$ for some $A=X^0=Y^0$, its {\it mapping cone} $M_f^{\mr}\in\CC$ is defined to be the complex
\[ X^1\xrightarrow{d_{M_f}^0}X^2\oplus Y^1\xrightarrow{d_{M_f}^1}X^3\oplus Y^2\xrightarrow{d_{M_f}^2}\cdots\xrightarrow{d_{M_f}^{n-1}}X^{n+1}\oplus Y^n\xrightarrow{d_{M_f}^n}Y^{n+1} \]
where $d_{M_f}^0=\begin{bmatrix}-d_X^1\\ f^1\end{bmatrix},$
$d_{M_f}^i=\begin{bmatrix}-d_X^{i+1}&0\\ f^{i+1}&d_Y^i\end{bmatrix}\ (1\le i\le n-1),$
$d_{M_f}^n=\begin{bmatrix}f^{n+1}&d_Y^n\end{bmatrix}$.

{\it The mapping cocone} is defined dually, for morphisms $h^{\mr}$ in $\CC$ satisfying $h^{n+1}=1$.
\end{definition}

\begin{definition}\label{definition:2.10}\cite[Definition 2.32]{HLN}
An {\it $n$-exangulated category} is a triplet $(\C,\E,\s)$ of additive category $\C$, additive bifunctor $\E\co\C\op\times\C\to\Ab$, and its exact realization $\s$, satisfying the following conditions.
\begin{itemize}[leftmargin=4em]
\item[{\rm (EA1)}] Let $A\ov{f}{\lra}B\ov{g}{\lra}C$ be any sequence of morphisms in $\C$. If both $f$ and $g$ are inflations, then so is $g f$. Dually, if $f$ and $g$ are deflations then so is $g f$.

\item[{\rm (EA2)}] For $\rho\in\E(D,A)$ and $c\in\C(C,D)$, let ${}_A\langle X^{\mr},c\uas\rho\rangle_C$ and ${}_A\Yr_D$ be distinguished $n$-exangles. Then $(1_A,c)$ has a {\it good lift} $f^{\mr}$, in the sense that its mapping cone gives a distinguished $n$-exangle $\langle M^{\mr}_f,(d_X^0)\sas\rho\rangle$.
\item[{\rm (EA2$\op$)}] Dual of {\rm (EA2)}.
\end{itemize}
Note that the case $n=1$, a triplet $(\C,\E,\s)$ is a  $1$-exangulated category if and only if it is an extriangulated category, see \cite[Proposition 4.3]{HLN}.
\end{definition}

\begin{example}
From \cite[Proposition 4.34]{HLN} and \cite[Proposition 4.5]{HLN},  we know that $n$-exact categories and $(n+2)$-angulated categories are $n$-exangulated categories.
There are some other examples of $n$-exangulated categories
 which are neither $n$-exact nor $(n+2)$-angulated, see \cite[Remark 4.39]{HLN}.
\end{example}

\begin{definition} \cite[Definition 3.2]{LZ}
Let $(\C,\E,\s)$ be an $n$-exangulated category.

(1) An object $P\in \C$ is called {\it projective} if, for any distinguished $n$-exangle
$$X^0\xrightarrow{d^0_X}X^1\xrightarrow{d^1_X}\cdots\xrightarrow{}X^{n-1}
\xrightarrow{d^{n-1}_X}X^n\xrightarrow{d^n_X}X^{n+1}\overset{\delta}{\dashrightarrow}$$
and any morphism $c$ in $\C(P, X^{n+1})$, there exists a morphism $b\in \C(P, X^n)$ satisfying
$d^n_{X}b = c$. We denote the full subcategory of projective objects in $\C$ by $\mathcal{P}$. Dually,
the full subcategory of injective objects in $\C$ is denoted by $\mathcal{I}$.
We denote by $\underline{\C}=\C/[\mathcal{P}]$ and $\overline{\C}=\C/[\mathcal{I}]$.

(2) We say that $\C$ has {\it enough projectives} if for any object $C\in\C$, there exists a distinguished $n$-exangle
$$A\xrightarrow{d_P^0}P^1\xrightarrow{d_P^1}\cdots\xrightarrow{}P^{n-1}
\xrightarrow{d_P^{n-1}}P^n\xrightarrow{d_P^n}C\overset{\delta}{\dashrightarrow}$$
satisfying $P^1 , P^2 ,\cdots, P^n\in \mathcal{P}$. We can define the notion of having {\it enough injectives}
dually.
\end{definition}

\section{\bf Auslander's defect formula over $n$-exangulated categories}
In this section, we introduce the defect of a distinguished $n$-exangle, which is an analogue to the defect of an $n$-exact sequence \cite[Definition 3.1]{JK}.  We prove an $n$-exangulated version of Auslander's defect formula. As an application, we obtain an $n$-exangulated version of  Auslander-Reiten duality formula, which provides a method to decide whether an $(n+2)$-angulated category has Auslander-Reiten $(n+2)$-angles.

In order to prove our main result Theorem \ref{thm4}, we need to make some preparations.

 We denote by $\mathfrak{s}$-${\rm def}\C$ the full subcategory of Mor$\C$ consisting of $\mathfrak{s}$-deflations.
   The full subcategory of $\mathfrak{s}$-${\rm def}\C$ consisting of split epimorphisms is denoted by s-epi$\C$.

 Recall that if
  $$X^{0}\xrightarrow{d_X^{0}} X^{1}\xrightarrow{d_X^{1}} \cdots\xrightarrow{} X^{n-1}\rightarrow X^{n}\xrightarrow{d_X^n} X^{n+1}\stackrel{\delta}{\dashrightarrow},$$
is a distinguished $n$-exangle, then the contravariant defect of $\delta$, denoted by $\widetilde{\delta}$, is defined by
$$\C(-, X^{n})\xrightarrow{\C(-, d_X^n)} \C(-, X^{n+1})\rightarrow \widetilde{\delta}\rightarrow 0 $$ the exact sequence of the functors. Dually, the covariant defect of $\delta$, denoted by $\check{\delta}$, is defined by the exact sequence of functors
$$\C(X^1, -)\xrightarrow{\C(d_X^0,-)}\C(X^{0},-)\rightarrow \check{\delta}\rightarrow 0.$$
It follows from \cite[Lemma 3.5(1)]{HLN} that $\check{\delta}=\mathrm{Ker}(\E(X^{n+1},- )\xrightarrow{(d_X^n)^*}\E(X^n, -))$.

Pick an object $X^{n}\xrightarrow{d_X^n} X^{n+1}$ of $\mathfrak{s}$-${\rm def}\C$. We have a distinguished $n$-exangle
   $$X^{0}\xrightarrow{d_X^{0}} X^{1}\xrightarrow{d_X^{1}} \cdots\xrightarrow{} X^{n-1}\xrightarrow{d_X^{n-1}} X^{n}\xrightarrow{d_X^n} X^{n+1}\stackrel{\delta}{\dashrightarrow},$$
   which induces the exact sequence
   $$\C(-, X^{0})\xrightarrow{\C(-, d_X^0)}\C(-, X^1)\rightarrow\cdots\rightarrow\C(-, X^{n})\xrightarrow{\C(-, d_X^n)} \C(-, X^{n+1})\rightarrow \widetilde{\delta}\rightarrow 0.$$  Note that $\widetilde{\delta}$ vanishes on projectives by  \cite[Lemma 3.4]{LZ}, and so $\widetilde{\delta}\in {\rm mod}\underline{\C}$ by \cite[Proposition 2.4]{L1}. In fact, the following sequence $\underline{\C}(-, X^{n})\xrightarrow{\underline{\C}(-, \underline{d}_X^n)} \underline{\C}(-, X^{n+1})\rightarrow \widetilde{\delta}\rightarrow 0$ is also exact by \cite[Lemma 1.27]{INP}, and so $\widetilde{\delta}\in {\rm mod}\underline{\C}$.
\begin{lemma}\label{prop3} Let $(\C, \mathbb{E}, \mathfrak{s})$ be an $n$-exangulated category with enough projectives. Then there is a functor $\Theta:\mathfrak{s}$-{\rm def} $\C\longrightarrow \mathrm{mod}\underline{\C}$.
\begin{proof}
We define $\Theta:\mathfrak{s}$-${\rm def}$$\C\longrightarrow \mathrm{mod}\underline\C$ by setting $\Theta(X^{n}\xrightarrow{d_X^n}X^{n+1})=\widetilde{\delta}$ by the above statement. We now show that $\Theta$ is well-defined.
Let $d_X^n: X^n\rightarrow X^{n+1}$ be an $\mathfrak{s}$-deflation.
Then we have an distinguished $n$-exangle
$$X^{0}\xrightarrow{d_X^{0}} X^{1}\xrightarrow{d_X^{1}} \cdots\xrightarrow{} X^{n-1}\xrightarrow{d_X^{n-1}} X^{n}\xrightarrow{d_X^n} X^{n+1}\stackrel{\delta}{\dashrightarrow},$$
   which induces the exact sequence
$\underline{\C}(-, X^{n})\xrightarrow{\underline{\C}(-, \underline{d}_X^n)} \underline{\C}(-, X^{n+1})\rightarrow \widetilde{\delta}\rightarrow 0$.
Thus $\Theta(X^{n}\xrightarrow{d_X^n}X^{n+1})=\mathrm{Coker}(\underline{\C}(-, X^{n})\xrightarrow{\underline{\C}(-, \underline{d}_X^n)} \underline{\C}(-, X^{n+1})).$ Hence we deduce that the definition of $\Theta$ is independent of the choice of the distinguished $n$-exangles whose $\mathfrak{s}$-deflation is $X^{n}\xrightarrow{d^n_X}X^{n+1}$.

Now let $X^n\xrightarrow{d_X^n}X^{n+1}$ and $Y^n\xrightarrow{d_Y^n}Y^{n+1}$ be two objects of $\mathfrak{s}$-${\rm def}\C$ and consider a commutative diagram
$$\xymatrix{
    X^{n} \ar[d]_{f^n} \ar[r]^{d_X^n} &X^{n+1} \ar[d]^{f^{n+1}} \\
    Y^n \ar[r]^{d_Y^n} & Y^{n+1}.}$$
    Therefore $(f^{n}, f^{n+1})$ lifts to the following morphism of distinguished $n$-exangles by the dual of \cite[Proposition 3.6(1)]{HLN}
 $$\xymatrix{
 X^{0}\ar[r]^{d_X^{0}}\ar@{-->}[d]_{f^{0}}&X^{1}\ar[r]^{d_X^{1}}\ar@{-->}[d]_{f^{1}}&X^{2}\ar[r]^{d_X^{2}}\ar@{-->}[d]_{f^{2}}&\cdots\ar[r]^{}
 &X^{n-1}\ar[r]^{d_X^{n-1}}\ar@{-->}[d]_{f^{n-1}}&X^{n}\ar[r]^{d_X^n}\ar[d]_{f^{n}}&X^{n+1}\ar@{-->}[r]^{\delta}\ar[d]_{f^{n+1}}&
\\
 Y^{0}\ar[r]^{d_Y^{0}}&Y^{1}\ar[r]^{d_Y^{1}}&Y^{2}\ar[r]^{d_Y^{2}}&\cdots\ar[r]^{}&Y^{n-1}\ar[r]^{d_Y^{n-1}}&Y^{n}\ar[r]^{d_Y^n}&Y^{n+1}\ar@{-->}[r]^{\rho}&.}$$  Hence we have the following commutative diagram with exact rows
 $$\xymatrix{
\underline{\C}(-, X^{n})\ar[r]^{\underline{\C}(-, \underline{d}_X^n)}\ar[d]_{\underline{\C}(-, \underline{f}^n)}&\underline{\C}(-, X^{n+1})\ar[r]^{}\ar[d]^{\underline{\C}(-, \underline{f}^{n+1})}&\widetilde{\delta}\ar[r]\ar[d]^{\widetilde{f}}&0
\\
\underline{\C}(-, Y^{n})\ar[r]^{\underline{\C}(-, \underline{d}_Y^n)}&\underline{\C}(-, Y^{n+1})\ar[r]^{}&\widetilde{\rho}\ar[r]&0.}$$
We set $\Theta(f^n,f^{n+1})=\widetilde{f}$.
Since $\Theta(X^{n}\xrightarrow{d_X^n}X^{n+1})=\widetilde{\delta}=\mathrm{Coker}(\underline{\C}(-, X^{n})\xrightarrow{\underline{\C}(-, \underline{d}_X^n)} \underline{\C}(-, X^{n+1}))$ and
$\Theta(Y^{n}\xrightarrow{d_Y^n}Y^{n+1})=\widetilde{\rho}=\mathrm{Coker}(\underline{\C}(-, Y^{n})\xrightarrow{\underline{\C}(-, \underline{d}_Y^n)} \underline{\C}(-, Y^{n+1}))$,
the universal property of cokernel implies that $\widetilde{f}$ is independent of the lifting morphisms $\{f^{i}\}_{1\leq i\leq n}$. This completes the proof.
\end{proof}

\end{lemma}

\begin{lemma}\label{prop2} Let $(\C, \mathbb{E}, \mathfrak{s})$ be an $n$-exangulated category.

{\rm (1)} If $\C$ has enough projectives, then $\underline{\C}$ has weak kernels and $\mathrm{mod}\underline{\C}$ forms an abelian category.

{\rm (2)} If $\C$ has enough injectives, then $\overline{\C}$ has weak cokernels and $\mathrm{mod}(\overline{\C}^{op})$ forms an abelian category.
\end{lemma}
\begin{proof}
We only prove (1). Let $f: C\rightarrow D$ be a morphism in $\C$. Since $\C$ has enough projectives, we have the following commutative diagram

$$\xymatrix{
 TD\ar[r]^{d_X^0}\ar@{=}[d]_{}&X^1\ar[r]^{d_X^1}\ar[d]_{f^1}&X^{2}\ar[r]^{d_X^{2}}\ar[d]_{f^2}&\cdots\ar[r]&X^n\ar[r]^{d_X^{n}}\ar[d]_{f^{n}}&
 C\ar@{-->}[r]^{f^{*}\delta}\ar[d]_{f}&\\
 TD\ar[r]^{d_P^0}&P^{1}\ar[r]^{d_P^{1}}&P^{2}\ar[r]^{d_P^{2}}&\cdots\ar[r]&P^{n}\ar[r]^{d_P^n}&D\ar@{-->}[r]^{\delta}&}$$  such that $$X^{1}\xrightarrow{\tiny\begin{bmatrix}-d_X^{1}\\f^{1}\end{bmatrix}}X^{2}\oplus P^{1}\rightarrow\cdots\rightarrow X^{n}\oplus P^{n-1}\xrightarrow{\tiny\begin{bmatrix}-d_X^{n}&0\\f^{n}&d_P^{n-1}\end{bmatrix}}C\oplus P^{n}\xrightarrow{\tiny\begin{bmatrix}f&d_P^{n}\end{bmatrix}}D\stackrel{(d_X^0)_{*}\delta}{\dashrightarrow}$$
is a distinguished $n$-exangle with $P^{i}\in \P$ by (EA2). Hence we have an exact sequence
 $$\xymatrix{\C(-, X^{n}\oplus P^{n-1})\ar[r]&\C(-, C\oplus P^{n})\ar[r]&\C(-, D)\ar[r]&F\ar[r]&0}$$
By \cite[Lemma 1.27]{INP}, there exists an exact sequence
$$ \xymatrix@=3em{\underline{\C}(-, X^n)\ar[r]^{\underline{\C}(-, -\underline{d}_X^n)}&\underline{\C}(-, C)\ar[r]^{\underline{\C}(-, \underline{f})}&\underline{\C}(-, D)\ar[r]&F\ar[r] &0.}$$
Hence $\underline{f}: C\rightarrow D$ has weak kernels. Therefore $\mathrm{mod}\underline{\C}$ is an abelian category.
\end{proof}

\begin{proposition}\label{thm2} Let $(\C, \mathbb{E}, \mathfrak{s})$ be an $n$-exangulated category with enough projectives.
With the above notations, there exists an equivalence of abelian categories
\begin{center}
$\mathfrak{s}\text{-}\mathrm{def}\C/ [\mathrm{s}\text{-}\mathrm{epi}\C]\cong\mathrm{mod}\underline{\C}$.
\end{center}
\end{proposition}
\begin{proof}
Firstly, we show that $\Theta$ is dense. Pick $F\in$ $\mathrm{mod}\underline{\C}$. So there exists an exact sequence
$$ \underline{\C}(-, C)\xrightarrow{\underline{\C}(-, \underline{f}^{n+1})}\underline{\C}(-, D)\stackrel{}{\rightarrow}F\rightarrow 0.$$

Since $\C$ has enough projectives, we have the following commutative diagram
$$\xymatrix@C=1.2cm{
 TD\ar[r]^{d_X^0}\ar@{=}[d]_{}&X^1\ar[r]^{d_X^1}\ar[d]_{f^1}&X^{2}\ar[r]^{d_X^{2}}\ar[d]_{f^2}&\cdots\ar[r]&X^n\ar[r]^{d_X^{n}}\ar[d]_{f^{n}}&
 C\ar@{-->}[r]^{(f^{n+1})^{*}\delta}\ar[d]_{f^{n+1}}&\\
 TD\ar[r]^{d_P^0}&P^{1}\ar[r]^{d_P^{1}}&P^{2}\ar[r]^{d_P^{2}}&\cdots\ar[r]&P^{n}\ar[r]^{d_P^n}&D\ar@{-->}[r]^{\delta}&}$$  such that $$X^{1}\xrightarrow{\tiny\begin{bmatrix}-d_X^{1}\\f^{1}\end{bmatrix}}X^{2}\oplus P^{1}\rightarrow\cdots\rightarrow X^{n}\oplus P^{n-1}\xrightarrow{\tiny\begin{bmatrix}-d_X^{n}&0\\f^{n}&d_P^{n-1}\end{bmatrix}}C\oplus P^{n}\xrightarrow{\tiny\begin{bmatrix}f^{n+1}&d_P^{n}\end{bmatrix}}D\stackrel{(d_X^0)_{*}\delta}{\dashrightarrow}$$
is a distinguished $n$-exangle with $P^{i}\in \P$ by (EA2) for $1\leq i\leq n$. It is easy to see that $\Theta(\tiny\begin{bmatrix}f^{n+1}&d_P^{n}\end{bmatrix})=F$.

Secondly, we show that the functor $\Theta:\mathfrak{s}$-def$\C\longrightarrow \mathrm{mod}\underline{\C}$, given by $X^{n}\xrightarrow{d_X^n}X^{n+1}\mapsto\mathrm{Coker}(\underline{\C}(-, X^{n})\xrightarrow{\C(-, \underline{d}_X^{n})} \underline{\C}(-, X^{n+1}))$, is full.
 Let $d_X^{n}: X^{n}\xrightarrow {}X^{n+1}$ and $d_Y^{n}: Y^n\xrightarrow{} Y^{n+1}$ be two objects of $\mathfrak{s}$-${\rm def}\C$ with $\Theta(d_X^n)=F$ and $\Theta(d_Y^n)=G$. Let $\eta:F\rightarrow G$ be a morphism in $\mathrm{mod}\underline{\C}$.

By assumption, we have the following commutative diagram in ${\rm mod}\underline{\C}$
$$\xymatrix@=3em{\underline{\C}(-, X^n)\ar@{-->}[d]_{\underline{\C}(-, \underline{g})}\ar[r]^{\underline{\C}(-, \underline{d}_X^n)}&\underline{\C}(-, X^{n+1} )\ar@{-->}[d]^{\underline{\C}(-, \underline{f}^{n+1})}\ar[r]&F\ar[r]\ar[d]^\eta&0\\
\underline{\C}(-, Y^n)\ar[r]^{\underline{\C}(-, \underline{d}_Y^n)}&\underline{\C}(-, Y^{n+1})\ar[r]&G\ar[r]&0,}$$
where the existence of dashed maps is guaranteed by the projectivity of representable functors.
By the Yoneda lemma, we have $\underline{f}^{n+1}\underline{d}_X^n=\underline{d}_Y^n\underline{g}$. Hence $f^{n+1}d_X^n-d_Y^ng$ factors through a projective object in $\C$, it is easy to check that there is a morphism $h: X^n\rightarrow Y^n$ such that $f^{n+1}d_X^n-d_Y^ng=d_Y^nh$ since $d_Y^n$ is an $\mathfrak{s}$-deflation.  Let $f^n:=g+h$. Then $f^{n+1}d_X^n=d_Y^nf^n$, and $(f^n, f^{n+1}): d_X^n\rightarrow d_Y^n$ a morphism in $\mathfrak{s}$-def$\C$. It is easy to see that $\Theta(f^n, f^{n+1})=\eta$, which implies that $\Theta$ is full.

Finally, let $X^{n}\xrightarrow{d_X^n}X^{n+1}$ be an object of $\mathfrak{s}$-${\rm def}\C$ with $\Theta(d_X^{n})=0$. We show that $d_X^n$ is a split epimorphism. By definition, $d_X^n$ extends to a distinguished $n$-exangle
$$X^{0}\xrightarrow{d_X^{0}} X^{1}\rightarrow \cdots\xrightarrow{} X^{n-1}\xrightarrow{d_X^{n-1}} X^{n}\xrightarrow{d_X^n} X^{n+1}\stackrel{\delta}{\dashrightarrow},$$
   which induces the exact sequence
   $\underline{\C}(-, X^{n})\xrightarrow{\underline{\C}(-,\underline{ d}_X^n)} \underline{\C}(-, X^{n+1})\rightarrow 0$ in ${\rm mod}\C$. In particular, we have an exact sequence $\underline{\C}(X^{n+1}, X^{n})\xrightarrow{\underline{\C}(X^{n+1},\underline{ d}_X^n)} \underline{\C}(X^{n+1}, X^{n+1})\rightarrow 0$. Hence there is $g: X^{n+1}\rightarrow X^n$ such that $1_{X^{n+1}}-d_X^ng$ factors through a projective object. It is easy to check that there exists a morphism $h: X^{n+1}\rightarrow X^n$ such that $1_{X^{n+1}}-d_X^ng=d_X^nh$ since $d_X^n$ is an $\mathfrak{s}$-deflation. Therefore $1_{X^{n+1}}=d_X^n(g+h)$,
it follows that $d_X^n$ is a split epimorphism by \cite[Claim 2.15]{HLN}.
Suppose that $(f^n, f^{n+1})$ is a morphism from $d_X^n: X^n\rightarrow X^{n+1}$ to $d_Y^n: Y^n\rightarrow Y^{n+1}$ with $d_X^n, d_Y^n\in\mathfrak{s}$-${\rm def}\C$.
By the dual of \cite[Propsotion 3.6(1)]{HLN}, we have the following commutative diagram
$$\xymatrix{
 X^{0}\ar[r]^{d_X^{0}}\ar@{-->}[d]_{f^{0}}&X^{1}\ar[r]^{d_X^{1}}\ar@{-->}[d]_{f^{1}}&X^{2}\ar[r]^{d_X^{2}}\ar@{-->}[d]_{f^{2}}&\cdots\ar[r]^{}
 &X^{n-1}\ar[r]^{d_X^{n-1}}\ar@{-->}[d]_{f^{n-1}}&X^{n}\ar[r]^{d_X^n}\ar[d]_{f^{n}}&X^{n+1}\ar@{-->}[r]^{\delta}\ar[d]_{f^{n+1}}&
\\
 Y^{0}\ar[r]^{d_Y^{0}}&Y^{1}\ar[r]^{d_Y^{1}}&Y^{2}\ar[r]^{d_Y^{2}}&\cdots\ar[r]^{}&Y^{n-1}\ar[r]^{d_Y^{n-1}}&Y^{n}\ar[r]^{d_Y^n}&Y^{n+1}\ar@{-->}[r]^{\rho}&.}$$
with $(f^{0})_{*}\delta=(f^{n+1})^{*}\rho.$
If $\Theta(f^n, f^{n+1})=0$, then the following diagram is commutative
$$\xymatrix@=3em{\underline{\C}(-, X^{n-1})\ar[d]_{\underline{\C}(-, \underline{f}^{n-1})}\ar[r]^{\underline{\C}(-, \underline{d}_X^{n-1})}&\underline{\C}(-, X^n)\ar[d]_{\underline{\C}(-, \underline{f}^n)}\ar[r]^{\underline{\C}(-, \underline{d}_X^n)}&\underline{\C}(-, X^{n+1} )\ar[d]^{\underline{\C}(-, \underline{f}^{n+1})}\ar[r]&\widetilde{\delta}\ar[r]\ar[d]^0&0\\
\underline{\C}(-, Y^{n-1})\ar[r]^{\underline{\C}(-, \underline{d}_Y^{n-1})}&\underline{\C}(-, Y^n)\ar[r]^{\underline{\C}(-, \underline{d}_Y^n)}&\underline{\C}(-, Y^{n+1})\ar[r]&\widetilde{\rho}\ar[r]&0}$$
and each row is exact. There exists a morphism $\underline{\C}(-, \underline{g}): \underline{\C}(-, X^{n+1})\rightarrow \underline{\C}(-, Y^n)$ such that $\underline{\C}(-, \underline{d}_Y^n)\C(-, \underline{g})=\C(-, \underline{f}^{n+1})$, that is, $\underline{f}^{n+1}=\underline{d}_Y^n\underline{g}$. Therefore, there is a morphism $g': X^{n+1}\rightarrow Y^n$ such that $f^{n+1}-d_Y^ng=d_Y^ng'$ because $f^{n+1}-d_Y^ng=d_Y^ng'$ factors through a projective object and $d_Y^n$ is an $\mathfrak{s}$-deflation. Let $h:=g+g'$. Then $f^{n+1}=d_Y^nh$.
 Since $d_Y^n(f^{n}-hd_X^n)=d_Y^{n}f^{n}-f^{n+1}d_X^{n}=0,$
 there exists a morphism $s:X^{n}\rightarrow Y^{n-1}$ with $f^{n}-hd_X^n=d_Y^{n-1}s$ by the exactness of $$\C(X^{n},Y^{n-1})\xrightarrow{\C(X^{n},d_Y^{n-1})}\C(X^{n},Y^{n})\xrightarrow{\C(X^{n},d_Y^n)}\C(X^{n},Y^{n+1}).$$
Hence we have the following commutative diagram
 $$\xymatrix{
    X^{n} \ar[d]_{\tiny\begin{pmatrix}1\\d_{X}^{n}\end{pmatrix}} \ar[r]^{d_{X}^{n}} & X^{n+1} \ar[d]^{1} \\
    X^{n}\oplus X^{n+1}\ar[d]_{(d_{Y}^{n-1}s,h)} \ar[r]^{(0,1)} & X^{n+1}\ar[d]^{f^{n+1}}\\
    Y^{n}\ar[r]^{d_{Y}^{n}} & Y^{n+1}}$$
 Since $X^{n}\rightarrow 0$ and $1:X^{n+1}\rightarrow X^{n+1}$ are $\mathfrak{s}$-deflations by \cite[Proposition 2.14]{H} and (R2) respectively, $(0,1):X^{n}\oplus X^{n+1}\rightarrow X^{n+1}$ is an $\mathfrak{s}$-deflation by \cite[Proposition 3.3]{HLN}.
 So $(0,1):X^{n}\oplus X^{n+1}\rightarrow X^{n+1}$ is an object of $\mathfrak{s}$-$def(\C)$.
 It follows that $(f^n, f^{n+1})$ factors through  $X^{n}\oplus X^{n+1}\xrightarrow{(0,1)} X^{n+1}$.
 Hence $\Theta$ is an objective functor.
 Therefore, $\mathfrak{s}\text{-}\mathrm{def}\C/ [\mathrm{s}\text{-}\mathrm{epi}\C]\cong\mathrm{mod}\underline{\C}$.
\end{proof}

Next, we define a functor $\Phi:\mathfrak{s}$-$def$$\C\longrightarrow \mathrm{mod}(\C^{op})$ by $\Phi(X^{n}\xrightarrow{d_X^n} X^{n+1})=\check{\delta}$, where $d_X^n$ extends to a distinguished $n$-exangle $X^{0}\xrightarrow{d_X^{0}} X^{1}\xrightarrow{d_X^{1}} \cdots\xrightarrow{} X^{n-1}\xrightarrow{d_X^{n-1}} X^{n}\xrightarrow{d_X^n} X^{n+1}\stackrel{\delta}{\dashrightarrow}$. Note that since $d_X^n$ is an $\mathfrak{s}$-deflation, $\check{\delta}$ vanishes on injectives by the dual of \cite[Lemma 3.4]{LZ}. So $\check{\delta}$ is an object of $\mathrm{mod}(\overline{\C}^{op})$ by the dual of \cite[Proposition 2.4]{L1}.

\begin{proposition}\label{thm3} Let $(\C, \mathbb{E}, \mathfrak{s})$ be an $n$-exangulated category with enough injectives. With the above notations,
$\Phi:\mathfrak{s}$-${\rm def}$$\C\longrightarrow \mathrm{mod}(\overline{\C}^{op})$ induces an equivalence of abelian categories
\begin{center}
$\mathfrak{s}\text{-}\mathrm{def}\C/ [\mathrm{s}\text{-}\mathrm{epi}\C]\cong (\mathrm{mod}(\overline{\C}^{op}))^{op}$.
\end{center}
\end{proposition}
\begin{proof}
The proof is similar to that of Proposition \ref{thm2}.
\end{proof}

\begin{remark}\label{rem:Duality} \rm{Let $(\C, \mathbb{E}, \mathfrak{s})$ be an $n$-exangulated category with enough projectives and injectives. In what follows, we let $\Sigma=\Phi\Theta^{-1}:\mathrm{mod}\underline{\C}\xrightarrow{\Theta^{-1}}\mathfrak{s}$-${\rm def}$/[{\rm s-epi}$\C$]$\xrightarrow {\Phi}\mathrm{mod}(\overline{\C}^{op})$ the duality between $\mathrm{mod}\underline{\C}$ and $\mathrm{mod}(\overline{\C}^{op})$.

}
\end{remark}

The following example is the $n$-exangulated version of \cite[Example 4.14]{L1}.
\begin{example}\label{ex1} {\rm(1)~Let $TX\xrightarrow{d_P^0} P^{1}\xrightarrow{d_P^1} P^{2}\xrightarrow{}\cdots\rightarrow P^{n}\xrightarrow{d_P^n} X\stackrel{\delta}{\dashrightarrow}$ be a distinguished $n$-exangle with $P^{i}\in \P$. Then $\widetilde{\delta}=\underline{\C}(-,X)$ and $\check{\delta}=\E(X,-)$.

(2) Let $X\xrightarrow{d_I^0} I^{1}\xrightarrow{d_I^1} I^{2}\xrightarrow{}\cdots\rightarrow I^{n}\xrightarrow{d_I^n} SX\stackrel{\delta}{\dashrightarrow}$ be a distinguished $n$-exangle with $I^{i}\in \I$. Then $\widetilde{\delta}=\E(-,X)$ and $\check{\delta}=\overline{\C}(X,-)$.
}
\end{example}
\begin{proof}
(1) follows from \cite[Lemma 3.4]{LZ} and the definition of projective objects. The proof of (2) is dual to that of (1).
\end{proof}

\begin{proposition}\label{prop4} Let $(\C, \mathbb{E}, \mathfrak{s})$ be an $n$-exangulated category with enough projectives and injectives. Then there exists a duality $\Sigma:\mathrm{mod}\underline{\C}\longrightarrow \mathrm{mod}(\overline{\C}^{op})$ and for every distinguished $n$-exangle $\delta$, $\Sigma(\widetilde{\delta})=\check{\delta}$. In particular, for any object $X$ in $\C$, we have

{\rm (1)} $\Sigma(\underline{\C}(-,X))=\E(X,-)$.

{\rm (2)} $\Sigma({\E(-,X)})=\overline{\C}(X,-)$.
\end{proposition}

\begin{proof}We only prove $(1)$.
By the definitions of $\Phi$ and $\Theta$, it is easy to see that $\Sigma(\widetilde{\delta})=\check{\delta}$ for every distinguished $n$-exangle $\langle X^{\mr},\delta\rangle$.
Let $X\in \C$.
Since $\C$ has enough projectives, there exists a distinguished $n$-exangle $TX\xrightarrow{d_P^0} P^{1}\xrightarrow{d_P^1} P^{2}\xrightarrow{}\cdots\rightarrow P^{n}\xrightarrow{d_P^n} X\stackrel{\delta}{\dashrightarrow}$
with $P^{i}\in\mathcal{P}$ for $1\leq i\leq n$.
Then $\widetilde{\delta}=\underline{\C}(-,X)$ and $\check{\delta}=\E(X,-)$ by Example \ref{ex1}.
Hence $\Sigma(\underline{\C}(-,X))=\Phi\Theta^{-1}(\underline{\C}(-,X))=\Phi\Theta^{-1}\widetilde{\delta}=\Phi(P^{n}\xrightarrow{d_P^n}X)=\E(X,-)=\check{\delta},$
where the last two equalities are due to the definitions of $\Phi$ and $\Theta$.
Similarly, we can show that $\Sigma({\E(-,X)})=\overline{\C}(X,-)$ for any $X\in \C$.
\end{proof}

As an application, we obtain an $n$-exangulated category version of Hilton-Ress Theorem, see \cite[Theorem 7]{M}.

\begin{theorem}\label{thm:Hilton-Ress} Let $(\C, \mathbb{E}, \mathfrak{s})$ be an $n$-exangulated category with enough projectives and injectives.

{\rm (1)} There is an isomorphism between $\underline{\C}(X,Y)$ and the group of the natural transformations from $\E(X,-)$ to $\E(Y,-)$.

{\rm (2)} There is an isomorphism between $\overline{\C}(X,Y)$ and the group of the natural transformations from $\E(-,X)$ to $\E(-,Y)$.
\end{theorem}
\begin{proof}
(1) follows from the duality of the part (1) of the above proposition in view of Yoneda lemma. (2) follows similarly.
\end{proof}
\begin{remark} In Theorem \ref{thm:Hilton-Ress}, when $n=1$, it is just the Proposition 4.8 in \cite{INP}. Moreover, if $\C$ is an $n$-cluster tilting subcategory of the category of finitely generated left $\Lambda$-modules over an Artin algebra $\Lambda$, then Theorem \ref{thm:Hilton-Ress} is just the Theorem 7.6 in \cite{AHK}.
\end{remark}

\begin{definition}\cite{AR1}
Let $k$ be a commutative Artin ring and $E$ the injective envelope of $k$.
Set $D = \mathrm{Hom}_{k}(-,E)$. A $k$-linear additive category $\C$ is called {\it dualizing $k$-variety} if the following conditions hold.

(1) $\C$ is $k$-linear, Hom-finite and Krull-Schmidt.

(2) For any $F\in \mathrm{mod}\C$, we have $DF\in \mathrm{mod}(\C^{op})$.

(3) For any $G\in \mathrm{mod}(\C^{op})$ , we have $DG\in \mathrm{mod}\C$.

In this case, we have a duality $D: \mathrm{mod}\C \rightarrow \mathrm{mod}(\C^{op}).$
\end{definition}

\begin{example}\cite[Example 2.5]{L1}\label{ex:dualizing variety}
(1) Let $A$ be an Artin $k$-algebra. Denote by mod-$A$ the category of finitely presented right $A$-modules, and by proj-$A$ the full subcategory of mod-$A$ formed by projective $A$-modules. Then both mod-$A$ and proj-$A$ are dualizing $k$-varieties.

(2) Any functorially finite subcategory of a dualizing $k$-variety is also a dualizing $k$-variety.

(3) Let $\C$ be a dualizing $k$-variety and $\mathscr{D}$ be a contravariantly (or covariant) finite subcategory. Then
$\C/[\mathscr{D}]$ is a dualizing $k$-variety.
\end{example}

 We have the following $n$-exangulated version of Auslander's defect formula.
\begin{theorem}\label{thm4}
Let $(\C, \mathbb{E}, \mathfrak{s})$ be an $n$-exangulated category with enough projectives and injectives. Assume that $\C$ is a dualizing $k$-variety. Then there is an equivalence $\tau:\underline{\C}\rightarrow \overline{\C}$ satisfying the following properties:
  $$D\check{\delta}=\widetilde{\delta}\tau^{-1},~~D\widetilde{\delta}=\check{\delta}\tau.$$
\end{theorem}
\begin{proof}
Since $\C$ is a dualizing $k$-variety,  $\overline{\C}$ is a dualizing $k$-variety by Example \ref{ex:dualizing variety} (3). The composition of $\Sigma:\mathrm{mod}\underline{\C}\longrightarrow \mathrm{mod}(\overline{\C}^{op})$ and $D:\mathrm{mod}(\overline{\C}^{op})\rightarrow\mathrm{mod}\overline{\C}$ define an equivalence $\Omega:\mathrm{mod}\underline{\C}\xrightarrow{\Sigma} \mathrm{mod}(\overline{\C}^{op})\xrightarrow{D}\mathrm{mod}\overline{\C}$.
Note that projective objects in $\mathrm{mod}\overline{\C}$ are representable functors as $\overline{\C}$ is a dualizing $k$-variety.
Then $\Omega (\underline{\C}(-,X))=D\E(X,-)\cong\overline{\C}(-,Y)$ for some $Y\in \C$ by Proposition \ref{prop4}.
Therefore, there is an equivalence $\tau:\underline{\C}\rightarrow \overline{\C}$ mapping $X$ to $Y$.
 Then the functor $\tau$ induces an equivalence $\tau_{*}^{-1}:\mathrm{mod}\underline{\C}\cong \mathrm{mod}\overline{\C}, F\mapsto F\tau^{-1}$, such that $\Omega=D\Sigma=\tau_{*}^{-1}$.

 Indeed, for $F\in \mathrm{mod}\underline{\C},$ we have an exact sequence
 $\underline{\C}(-,X^{n})\rightarrow \underline{\C}(-, X^{n+1})\rightarrow F\rightarrow 0.$
 Let $\tau(X^{n})=Y^{n}$ and $\tau(X^{n+1})=Y^{n+1}$.
 Then we have an exact sequence
 $\overline{\C}(-,Y^{n})\rightarrow \overline{\C}(-,Y^{n+1})\rightarrow \Omega F\rightarrow 0.$
 Since $\tau:\underline{\C}\rightarrow \overline{\C}$ is an equivalence,
 we have the following commutative diagram
 $$\xymatrix{\underline{\C}(\tau^{-1}(-),X^{n})\ar[r]\ar[d]^{\cong}&\underline{\C}(\tau^{-1}(-),X^{n+1})\ar[r]\ar[d]^{\cong}&F\tau^{-1}\ar[r]\ar@{=}[d] &0\\
 \overline{\C}(-,Y^{n})\ar[r]&\overline{\C}(-,Y^{n+1})\ar[r]&F\tau^{-1}\ar[r] &0
 }
 .$$
So without loss of generality, we could say that $\tau_{*}^{-1}F=F\tau^{-1}=\Omega F.$
Hence $\Omega=D\Sigma=\tau_{*}^{-1}$.

 Assume that $\delta$ is a distinguished $n$-exangle. Then $D\Sigma(\widetilde{\delta})=D\check{\delta}$ by Proposition \ref{prop4} and
   $\tau_{*}^{-1}(\widetilde{\delta})=\widetilde{\delta}\tau^{-1}$ by definition of $\tau_{*}^{-1}$.
  Hence we have $D\check{\delta}=\widetilde{\delta}\tau^{-1}$.
  So $D\widetilde{\delta}=\check{\delta}\tau$.
\end{proof}

 As an application of Theorem \ref{thm4}, we have the $n$-extriangulated version of  Auslander-Reiten duality formula.
\begin{corollary}\label{corollary:Auslander-Reiten duality}
Let $(\C, \mathbb{E}, \mathfrak{s})$ be an $n$-exangulated category with enough projectives and injectives. Assume that $\C$ is a dualizing $k$-variety. Then there is an equivalence $\tau:\underline{\C}\rightarrow \overline{\C}$ satisfying the following properties:
$$D\E(-,X)\cong\underline{\C}(\tau^{-1}X,-),~~ D\E(X,-)\cong\overline{\C}(-,\tau X).$$
\end{corollary}
\begin{proof}
For a distinguished $n$-exangle $X\xrightarrow{d_I^0} I^{1}\xrightarrow{d_I^1} I^{2}\xrightarrow{d_I^2}\cdots\rightarrow I^{n}\xrightarrow{d_I^n} SX\stackrel{\delta}{\dashrightarrow}$  with $I^{i}\in \I$ for $1\leq i\leq n$, then $\widetilde{\delta}=\E(-,X)$ and $\check{\delta}=\overline{\C}(X,-)$ by Example \ref{ex1}.
Since $D\widetilde{\delta}=\check{\delta}\tau$ by Theorem \ref{thm4},  $D\E(-,X)\cong\overline{\C}(X,\tau-)\cong\underline{\C}(\tau^{-1}X,-).$
Similarly, we can show $D\E(X,-)\cong\overline{\C}(-,\tau X).$
\end{proof}

\begin{remark}
In Theorem \ref{thm4} and Corollary \ref{corollary:Auslander-Reiten duality}, when $\C$ is an exact category, it is just the Theorem 4.4 in \cite{L1}.
\end{remark}

\begin{corollary}\label{corollary:3.14}
Let $\C$ be a Krull-schmidt $(n+2)$-angulated category with an $n$-suspension functor $\Sigma^{n}$. If $\C$ is a dualizing $k$-variety. Then $\C$ has Auslander-Reiten $(n+2)$-angles.
\end{corollary}
\begin{proof}
It follows from Corollary \ref{corollary:Auslander-Reiten duality} and \cite[Theorem 4.5]{Z}.
\end{proof}

\begin{definition}\cite[Definition 5.19]{HLN}
Let $(\C, \mathbb{E}, \mathfrak{s})$ be an $n$-exangulated category  and $\mathcal{T}\subseteq\C$  a full additive subcategory closed by isomorphisms and direct summands.
 Such $\mathcal{T}$ is called an {\it $n$-cluster tilting subcategory} of $\C$, if it
satisfies the following conditions.

(1) $\mathcal{T}\subseteq\C$ is functorially finite.

(2) For any $C\in\C$, the following are equivalent.

(i) $C\in \mathcal{T}$; ~(ii) $\E^{i}(C,T)=0$ for any $1\leq i\leq n-1$;~
(iii) $\E^{i}(T,C)=0$ for any $1\leq i\leq n-1$.
\end{definition}

Consider the following conditions.

\begin{condition}\cite[Condition 5.21]{HLN} \label{cond:1}
 Let $\mathcal{T}\subseteq\C$ be an $n$-cluster tilting subcategory.

\emph{\rm (c1)} If $C\in\C$ satisfies $\E^{n-1}(T,C) = 0$, then there is $\mathfrak{s}$-triangle $\xymatrix{D\ar[r]^{}&P\ar[r]^{}&C\ar@{-->}[r]^{\delta}&,}$
with $P\in \mathcal{P}$ for which
$\C(T,q): \C(T,D)\rightarrow\C(T,P)$
is injective for any $T\in\mathcal{T}$.

\emph{\rm (c2)} Dually, if $A\in C$ satisfies $\E^{n-1}(A,T) = 0$, then there is $s$-triangle
$\xymatrix@C=0.7cm{A\ar[r]^{}&I\ar[r]^{}&S\ar@{-->}[r]^{\delta}&,}$
with $I\in \mathcal{I}$
for which
$\C(j,T): \C(S,T)\rightarrow \C(I,T)$
is injective for any $T\in \mathcal{T}$.
\end{condition}

\begin{remark}
(1) If $\C$ is an exact category with enough projectives and injectives, then Condition \ref{cond:1} is trivially satisfied by any $n$-cluster tilting subcategory $\mathcal{T}\subseteq\C$.
\end{remark}

\begin{condition} \cite[Condition 5.23]{HLN}\label{cond:2}
 Let $(\C, \mathbb{E}, \mathfrak{s})$ be an extriangulated category. Consider the following conditions.

\emph{(1)} Let $f\in\mathcal{C}(A, B), g\in\mathcal{C}(B, C)$ be any composable pair of morphisms. If $gf$ is an $\mathfrak{s}$-inflation, then so is $f$.

\emph{(2)} Let $f\in\mathcal{C}(A, B), g\in\mathcal{C}(B, C)$ be any composable pair of morphisms. If $gf$ is an $\mathfrak{s}$-deflation, then so is $g$.
\end{condition}

As a consequence of Theorem \ref{thm4} and Corollary \ref{corollary:Auslander-Reiten duality}, we have the following corollary.

\begin{corollary}\label{corollary:3.15}
Let $(\C, \mathbb{E}, \mathfrak{s})$ be an extriangulated category satisfying Condition \ref{cond:2} with enough projectives and injectives and its $n$-cluster tilting subcategory $\mathcal{M}$ satisfies Condition \ref{cond:1}. Assume that either $\underline{\mathcal{M}}$ or $\overline{\mathcal{M}}$ is a dualizing $k$-variety. Then there is an equivalence $\tau:\underline{\mathcal{M}}\rightarrow \overline{\mathcal{M}}$ satisfying the following properties:

{\rm (1)} $D\check{\delta}=\widetilde{\delta}\tau^{-1},$ $D\widetilde{\delta}=\check{\delta}\tau;$

{\rm (2)} $D\E^{n}(Y,X)\cong\underline{\mathcal{M}}(\tau^{-1}X,Y),$ $D\E^{n}(X,Y)\cong\overline{\mathcal{M}}(Y,\tau X)$ for any $X\in \mathcal{M}$ and $Y\in \mathcal{M}.$

\end{corollary}
\begin{proof}
Since $(\mathcal{T},\E^{n},\mathfrak{s}^{n})$ is an $n$-exangulated category by \cite[Theorem 5.39]{HLN}, this corollary follows from Theorem \ref{thm4} and Corollary \ref{corollary:Auslander-Reiten duality}.
\end{proof}
\begin{remark} In Theorem \ref{corollary:3.15}, when $\C$ is the category of finitely generated left $\Lambda$-modules over an Artin algebra $\Lambda$, it is just the Theorem 3.8 in \cite{JK} or Lemma 3.2 in \cite{I}.
\end{remark}

\section{Classifying substructures of $n$-exangulated categories}\label{section3}

In the following, we always assume that $\C:=(\C, \mathbb{E}, \mathfrak{s})$ is a skeletally small $n$-exangulated category with weak-kernels.

Recall that if $\langle X^{\mr}, \delta\rangle$ is a distinguished $n$-exangle in  $(\C, \mathbb{E}, \mathfrak{s})$, then the contravariant defect of $\delta$, denoted by $\widetilde{\delta}$, is defined by  $\C(-, X^n)\xrightarrow{\C(-, d_X^n)} \C(-, X^{n+1})\rightarrow\widetilde{\delta}\rightarrow 0$ the exact sequence of functors in ${\rm mod}\C$. We denote by ${\rm def}\mathbb{E}$ the full subcategory in ${\rm mod}\C$ consisting of all functors isomorphic to contravariant defects.

\begin{proposition} The subcategory ${\rm def}\mathbb{E}$ is closed under taking kernels and cokernels in ${\rm mod}\C$.
\end{proposition}

\begin{proof} Let $\langle X^{\mr}, \delta\rangle$ be a distinguished $n$-exangle in  $\C$ and $L$ a submodule of $\widetilde{\delta}$ in ${\rm mod}\C$. Then there is a surjection $p: \C(-, Y^{n+1})\rightarrow L$ since $L\in {\rm mod}\C$. Thus we obtain  the following commutative diagram in ${\rm mod}\C$

$$\xymatrix@=2.8em{&&&\C(-, Y^{n+1})\ar@{-->}[d]\ar[r]^{\qquad p}&L\ar[r]\ar@{^{(}->}[d]^{\alpha}&0\\
\C(-, X^0)\ar[r]^{\ \ \ \C(-, d_X^0)}&\cdots\ar[r]&\C(-, X^n)\ar[r]^{\C(-, d_X^n)\ \ }&\C(-, X^{n+1})\ar[r]^{\qquad  \pi}&\widetilde{\delta}\ar[r]&0.}$$
Here the dashed map exists by the projectivity of $\C(-, Y^{n+1})$. Then by the Yoneda lemma, there is a corresponding $f^{n+1}: Y^{n+1}\rightarrow X^{n+1}$ in $\C$. Now assume that $Y^{\mr}$ is an $\mathfrak{s}$-conflation which realizes $(f^{n+1})^*\delta\in \mathbb{E}(Y^{n+1}, X^0)$, then $(1_{X^0}, f^{n+1}): (f^{n+1})^*\delta\rightarrow \delta$ has a good lift $f^{\mr}: \langle Y^{\mr}, (f^{n+1})^*\delta\rangle\rightarrow \langle X^{\mr}, \delta\rangle$ which makes $\langle M_{f^{\mr}}, (d_Y^0)_*\delta\rangle$ a distinguished $n$-exangle by (EA2), that is, we have the following commutative diagram

$$\xymatrix{X^0\ar[r]^{d_{Y}^{0}}\ar@{=}[d]&Y^1\ar[r]\ar[d]_{f^{1}}&\cdots \ar[r]&Y^{n-1}\ar[d]_{f^{n-1}}\ar[r]^{d_{Y}^{n-1}}& Y^{n}\ar[d]_{f^{n}}\ar[r]^{d_{Y}^{n}}& Y^{n+1} \ar[d]_{f^{n+1}} \ar@{-->}[r]^{\qquad(f^{n+1})^*\delta}&\\
   X^0\ar[r]^{d_{X}^{0}}&X^1\ar[r]&\cdots \ar[r]&X^{n-1}\ar[r]^{d_X^{n-1}}& X^{n}\ar[r]^{d_{X}^{n}}& X^{n+1} \ar@{-->}[r]^{\delta}&}$$
such that
$$Y^1\xrightarrow{\tiny\begin{pmatrix}-d_Y^1\\f^1\end{pmatrix}}Y^2\oplus X^1\to\cdots\to Y^n\oplus X^{n-1}\xrightarrow{\tiny\begin{pmatrix}-d_Y^n&0\\f^n&d_X^{n-1}\end{pmatrix}}Y^{n+1}\oplus X^n\xrightarrow{\tiny\begin{pmatrix}f^{n+1}&d_X^n\end{pmatrix}}X^{n+1}
\xymatrix{\ar@{-->}[r]^{(d_Y^0)_*\delta}&}$$
is a distinguished $n$-exangle. Thus we obtain the following commutative diagram in ${\rm mod}\C$
$$\xymatrix@=2.8em{\C(-, Y^n)\ar[d]_{\C(-, f^n)}\ar[r]^{\C(-, d_Y^n)}&\C(-, Y^{n+1})\ar[d]^{\C(-, f^{n+1})}\ar[r]^{\qquad p}&L\ar[r]\ar@{^{(}->}[d]^{\alpha}&0\\
\C(-, X^n)\ar[r]^{\C(-, d_X^n)}&\C(-, X^{n+1})\ar[r]^{\qquad  \pi}&\widetilde{\delta}\ar[r]&0}$$
where the bottom sequence is exact. It is easy to check that ${\rm Im}\C(-, d_Y^n)\subseteq {\rm Ker}p$ since $\alpha$ is an injection. Now we claim that
${\rm Ker}p\subseteq{\rm Im}\C(-, d_Y^n)$. Indeed, for any object $M\in\C$, we have the following commutative diagram
$$\xymatrix@=2.8em{\C(M, Y^n)\ar[d]_{\C(M, f^n)}\ar[r]^{\C(M, d_Y^n)}&\C(M, Y^{n+1})\ar[d]^{\C(M, f^{n+1})}\ar[r]^{\qquad p_M}&L(M)\ar[r]\ar@{^{(}->}[d]^{\alpha_M}&0\\
\C(M, X^n)\ar[r]^{\C(M, d_X^n)}&\C(M, X^{n+1})\ar[r]^{\qquad  \pi_M}&\widetilde{\delta}(M)\ar[r]&0.}$$
If $a\in\C(M, Y^{n+1})$ satisfies $p_M(a)=0$, then $\pi_M\C(M, f^{n+1})(a)=\alpha_M p_M(a)=0$. There is $b\in\C(M, X^n)$ such that $\C(M, d_X^n)(b)=\C(M, f^{n+1})(a)$, that is, $d_X^nb=f^{n+1}a$, which implies that ${\footnotesize\begin{pmatrix}f^{n+1}& d_X^n\end{pmatrix}\begin{pmatrix}-a\\b\end{pmatrix}}=d_X^nb-f^{n+1}a=0$. Hence there is a morphism ${\footnotesize\begin{pmatrix}a'\\b'\end{pmatrix}}: M\rightarrow Y^n\oplus X^{n-1}$ which makes the following diagram commutative
$$\xymatrix@=4em{&M\ar@{-->}[dl]_{\tiny\begin{pmatrix}a'\\b'\end{pmatrix}}\ar[d]^{\tiny\begin{pmatrix}-a\\b\end{pmatrix}}&\\
Y^n\oplus X^{n-1}\ar[r]^{\tiny\begin{pmatrix}-d_Y^n&0\\f^n&d_X^{n-1}\end{pmatrix}}&Y^{n+1}\oplus X^n\ar[r]^{\tiny\qquad \begin{pmatrix}f^{n+1}&d_X^n\end{pmatrix}}&X^{n+1}.}$$
Hence ${\footnotesize\begin{pmatrix}-a\\b\end{pmatrix}=\begin{pmatrix}-d_Y^n&0\\f^n&d_X^{n-1}\end{pmatrix}\begin{pmatrix}a'\\b'\end{pmatrix}
=\begin{pmatrix}-d_Y^na'\\f^na'+d_X^{n-1}b'\end{pmatrix}}$, and $a=d_Y^na'=\C(M, d_Y^n)(a')$ which implies that ${\rm Ker}p_M\subseteq {\rm Im}\C(M, d_Y^n)$. So ${\rm Ker}p\subseteq {\rm Im}\C(-, d_Y^n)$, hence $L\in {\rm def}\mathbb{E}$. It is clearly that ${\rm def}\mathbb{E}$ is closed under taking kernels.

Next we claim that ${\rm def}\mathbb{E}$ is closed under taking cokernels. Let $\langle X^{\mr}, \delta\rangle$ and $\langle Y^{\mr}, \rho\rangle$ be distinguished $n$-exangles. Consider a morphism $\alpha: \widetilde{\delta}\rightarrow \widetilde{\rho}$. We shall show that ${\rm Coker}\alpha$ still belongs to ${\rm def}\mathbb{E}$. By the Yoneda lemma, the morphism $\alpha$ induces a morphism between presentation of $\widetilde{\delta}$ and $\widetilde{\rho}$, and there is a morphism of distinguished $n$-exangles $f^{\mr}: \langle X^{\mr}, \delta\rangle\rightarrow \langle Y^{\mr}, \rho\rangle$ by the dual of \cite[Proposition 3.6(1)]{HLN}. In particular, $(f^0)_*\delta=(f^{n+1})^*\rho$. Assume that $Z^{\mr}$ is an $\mathfrak{s}$-conflation which realizes $(f^{n+1})^*\rho\in \mathbb{E}(X^{n+1}, Y^0)$, $(1_{Y^0}, f^{n+1}): (f^{n+1})^*\rho\rightarrow \rho$  has a good lift $g^{\mr}: \langle Z^{\mr}, (f^{n+1})^*\rho\rangle\rightarrow \langle Y^{\mr}, \rho\rangle$ which makes $\langle M_{g^{\mr}}, (d_Z^0)_*\rho\rangle$ a distinguished $n$-exangle by (EA2). It is easy to check that there are morphisms of distinguished $n$-exangles with the following commutative diagram

$$\xymatrix@=3.3em{X^0\ar[r]^{d_{X}^{0}}\ar[d]^{f^0}&X^1\ar[r]\ar@{-->}[d]^{h^{1}}&\cdots \ar[r]&X^{n-1}\ar@{-->}[d]^{h^{n-1}}\ar[r]^{d_X^{n-1}}& X^{n}\ar@{-->}[d]^{h^{n}}\ar[r]^{d_{X}^{n}}& X^{n+1} \ar@{=}[d] \ar@{-->}[r]^{\delta}&\\
   Y^0\ar@{=}[d]\ar[r]^{d_{Z}^{0}}&Z^1\ar[d]^{g^1}\ar[r]&\cdots \ar[r]&Z^{n-1}\ar[d]^{g^{n-1}}\ar[r]^{d_Z^{n-1}}& Z^{n}\ar[d]^{g^n}\ar[r]^{d_{Z}^{n}}& X^{n+1} \ar[d]^{f^{n+1}}\ar@{-->}[r]^{\ \ \ (f^{n+1})^*\rho}&\\
   Y^0\ar[d]^{d_Z^0}\ar[r]^{d_Y^0}&Y^1\ar[r]\ar[d]^{\tiny\begin{pmatrix}0\\1\end{pmatrix}}&\cdots\ar[r]&Y^{n-1}\ar[r]^{d_Y^{n-1}}
   \ar[d]^{\tiny\begin{pmatrix}0\\1\end{pmatrix}}&Y^{n}\ar[r]^{d_Y^{n}}
   \ar[d]^{\tiny\begin{pmatrix}0\\1\end{pmatrix}}&Y^{n+1}\ar@{-->}[r]^{\rho}\ar@{=}[d]&\\
  Z^1\ar[r]^{\tiny\begin{pmatrix}-d_Z^1\\g^1\end{pmatrix}\ \ \ \ }&Z^2\oplus Y^1\ar[r]&\cdots\ar[r]&Z^n\oplus Y^{n-1}\ar[r]^{\tiny\begin{pmatrix}-d_Z^n&0\\g^n&d_Y^{n-1}\end{pmatrix}}&X^{n+1}\oplus Y^n\ar[r]^{\tiny\qquad \begin{pmatrix}f^{n+1}&d_Y^n\end{pmatrix}}&Y^{n+1}\ar@{-->}[r]^{\qquad (d_Z^0)_*\rho}&}$$
Here the dashed maps $h^i$ exist by (R0). Thus we get the following commutative diagram
$$\xymatrix@C=5em{\C(-, X^n)\ar[d]^{\C(-, g^nh^n)}\ar[r]^{\C(-, d_X^n)}&\C(-, X^{n+1})\ar[d]^{\C(-, f^{n+1})}\ar[r]^{\qquad \pi}&\widetilde{\delta}\ar[r]\ar[d]^{\alpha}&0\\
\C(-, Y^n)\ar[r]^{\C(-, d_Y^n)}\ar[d]^{\C(-, {\tiny\begin{pmatrix}0\\1\end{pmatrix}})}&\C(-, Y^{n+1})\ar@{=}[d]\ar[r]^{\qquad  p}&\widetilde{\rho}\ar[d]^\beta\ar[r]&0\\
\C(-, X^{n+1}\oplus Y^{n-1})\ar[r]^{\ \ \ \ \ \C(-, {\tiny\begin{pmatrix}f^{n+1}&d_Y^n\end{pmatrix}})}&\C(-, Y^{n+1})\ar[r]^{\qquad  q}&\widetilde{(d_Z^0)_*\rho}\ar[r]&0}$$
in ${\rm mod}\C$ with exact rows. Now we claim that the sequence
$$\xymatrix@C=4.5em{\C(-, X^{n+1})\ar[r]^{\qquad p\C(-, f^{n+1})}&\widetilde{\rho}\ar[r]^\beta&\widetilde{(d_Z^0)_*\rho}\ar[r]&0}$$
is exact. Note that $\beta p=q$ is epic, then $\beta$ is epic.
$$\beta p\C(-, f^{n+1})=q\C(-, f^{n+1})=q \C(-, {\scriptsize\begin{pmatrix}f^{n+1}&d_Y^n\end{pmatrix}})\C(-, {\scriptsize\begin{pmatrix}1\\0\end{pmatrix}})=0.$$
Hence ${\rm Im}(p \C(-, f^{n+1}))\subseteq {\rm Ker}\beta$. Consider the sequence
$$\xymatrix@C=4.5em{\C(M, X^{n+1})\ar[r]^{\qquad p\C(M, f^{n+1})}&\widetilde{\rho}(M)\ar[r]^{\beta_M}&\widetilde{(d_Z^0)_*\rho}(M)\ar[r]&0}$$
for any $M\in\C$. If $m\in\widetilde{\rho}(M)$ satisfies $\beta_M(m)=0$, then there is a morphism $r\in\C(M, Y^{n+1})$ such that $p_M(r)=m$ because $p_M$ is epic. Hence
$q_M(r)=\beta_Mp_M(r)=\beta_M(m)=0$. Thus there is ${\scriptsize \begin{pmatrix}s\\t\end{pmatrix}}\in\C(M, X^{n+1}\oplus Y^n)$ such that $\C(M, {\scriptsize \begin{pmatrix}f^{n+1}&d_Y^n\end{pmatrix}}){\scriptsize \begin{pmatrix}s\\t\end{pmatrix}}=r$ since the sequence
$$\xymatrix@C=5em{\C(M, X^{n+1}\oplus Y^{n-1})\ar[r]^{\ \ \ \ \ \C(M, {\tiny\begin{pmatrix}f^{n+1}&d_Y^n\end{pmatrix}})}&\C(M, Y^{n+1})\ar[r]^{\qquad  q_M}&\widetilde{(d_Z^0)_*\rho}(M)\ar[r]&0}$$
is exact. So $\C(M, {\scriptsize \begin{pmatrix}f^{n+1}&d_Y^n\end{pmatrix}}){\scriptsize \begin{pmatrix}s\\t\end{pmatrix}}={\scriptsize \begin{pmatrix}f^{n+1}&d_Y^n\end{pmatrix}}{\scriptsize \begin{pmatrix}s\\t\end{pmatrix}}=f^{n+1}s+d_Y^nt=r$ and $p_M\C(M, f^{n+1})(s)=p_M(f^{n+1}s)=p_M(r-d_Y^nt)=p_M(r)-p_M(d_Y^nt)=m-p_M\C(M, d_Y^n)(t)=m$. Then ${\rm Ker}\beta_M\subseteq {\rm Im}(p_M\C(M, f^{n+1}))$ which implies that ${\rm Ker}\beta\subseteq {\rm Im}(p\C(-, f^{n+1}))$. Since $\alpha\pi=p\C(-, f^{n+1})$, the sequence
$$\xymatrix@C=3em{\C(-, X^{n+1})\ar[r]^{\ \ \ \ \ \alpha\pi}&\widetilde{\rho}\ar[r]^\beta&\widetilde{(d_Z^0)_*\rho}\ar[r]&0}$$
is exact which implies the sequence
$$\xymatrix@C=3em{\widetilde{\delta}\ar[r]^{\alpha}&\widetilde{\rho}\ar[r]^\beta&\widetilde{(d_Z^0)_*\rho}\ar[r]&0}$$
is exact because $\pi$ is epic, that is, ${\rm Coker}\alpha=\widetilde{(d_Z^0)_*\rho}\in{\rm def}\mathbb{E}$, as desired.
\end{proof}

We shall study the properties of effaceable functors and Auslander's defects in ${\rm mod}\C$ for  an $n$-exangulated category $(\C, \E, \mathfrak{s})$. The notion of effaceable functors in ${\rm mod}\C$ for an extriangulated category was introduced in \cite{Oga}.

\begin{definition}\label{defect}
Let  $(\C, \mathbb{E}, \mathfrak{s})$ be an $n$-exangulated category. An object $F\in {\rm mod}\C$ is said to be {\it effaceable} if it satisfies the following condition:

 For any $X\in \C$ and any $x\in F(X)$, there exists an $\mathfrak{s}$-deflation $\alpha: Y\rightarrow X$ such that $F(\alpha)(x)=0$.

 We denote by ${\rm eff}\mathbb{E}$ the full subcategory of effaceable functors in ${\rm mod}\C$.
\end{definition}

 Recall that a {\it Serre subcategory} of
an abelian category is a subcategory which is closed under taking subobjects, quotients and extensions. We have the following result by a similar argument with \cite[Lemma 2.3]{Oga} and \cite[Proposition 2.5]{Oga}.

\begin{proposition}\label{Pro1} Let  $(\C, \mathbb{E}, \mathfrak{s})$ be an $n$-exangulated category. Then we have an equality ${\rm eff}\mathbb{E}={\rm def}\mathbb{E}$. In particular, ${\rm def}\mathbb{E}$ is a Serre subcategory in ${\rm mod}\C$.
\end{proposition}

For an $n$-exangulated category $(\C,\mathbb{E},\mathfrak{s})$, the notion of closed subbifunctors of $\mathbb{E}$ was
introduced in \cite{HLN}. Let $(\C, \mathbb{E}, \mathfrak{s})$ be an $n$-exangulated category.
An additive subbifunctor $\mathbb{F}$
of $\E$ is called {\it  closed on the right} (resp. {\it closed on the left}) if
 $\mathbb{F}(-,X^0)\xLongrightarrow{(d_X^0)_{*}}\mathbb{F}(-,X^1)\xLongrightarrow{(d_X^1)_{*}}\mathbb{F}(-,X^2)
$ (resp. $\mathbb{F}(X^{n+1},- )\xLongrightarrow{(d_X^n)^{*}}\mathbb{F}(X^n,-)\xLongrightarrow{(d_X^{n-1})^{*}}\mathbb{F}(X^{n-1},-)
$)
is exact for any $\mathfrak{s}|_{\mathbb{F}}$-conflation $X^{\mr}$, where $\mathfrak{s}|_{\mathbb{F}}$ is the restriction of $\mathfrak{s}$ onto $\mathbb{F}$.
It follows from \cite[Lemma 3.15]{HLN} that $\mathbb{F}$ is closed on the right if and only if $\mathbb{F}$ is closed on the left for any additive subbifunctor $\mathbb{F}\subseteq \mathbb{E}$. In this case, we simply say $\mathbb{F}\subseteq \mathbb{E}$ is {\it closed}.
The most important property of closed subbifunctors is the following result.
\begin{proposition}\label{Pro}{\rm \cite[Proposition 3.16]{HLN}}
 Let $(\C, \mathbb{E}, \mathfrak{s})$ be an $n$-exangulated category  and  any additive subbifunctor $\mathbb{F}\subseteq \mathbb{E}$. The following are equivalent.

 {\rm (1)}  $(\C, \mathbb{F}, \mathfrak{s}|_\mathbb{F})$ is an $n$-exangulated;

 {\rm (2)} $\mathfrak{s}|_\mathbb{F}$-inflation is closed under composition;

 {\rm (3)} $\mathfrak{s}|_\mathbb{F}$-deflation is closed under composition;

{\rm (4)}  $\mathbb{F}\subseteq \mathbb{E}$ is closed.
\end{proposition}

The following result is the $n$-exangulated version of \cite[Lemma 3.1]{Eno}, however the proof in here is simpler than that of \cite[Lemma 3.1]{Eno}.
\begin{lemma} Let  $(\C, \mathbb{E}, \mathfrak{s})$ be an $n$-exangulated category and $\mathbb{F}$ a closed subbifunctor of $\mathbb{E}$. Suppose that we have two elements $\rho\in \mathbb{F}(Y^{n+1}, Y^0)$ and $\delta\in \mathbb{E}(X^{n+1}, X^0)$. If $\widetilde{\delta}$ and $\widetilde{\rho}$ are isomorphic, then $\delta\in \mathbb{F}(Y^{n+1}, Y^0)$.
\end{lemma}
\begin{proof} Let $\mathfrak{s}(\delta)=[ \xymatrix{X^0\ar[r]^{d_X^0}&X^1\ar[r]^{d_X^1}&\cdots\ar[r]&X^{n-1}\ar[r]^{d_X^{n-1}}&X^n\ar[r]^{d_X^n}&X^{n+1}} ]$ and

$\mathfrak{s}(\rho)=[\xymatrix{ Y^0\ar[r]^{d_Y^0}&Y^1\ar[r]^{d_Y^1}&\cdots\ar[r]&Y^{n-1}\ar[r]^{d_Y^{n-1}}&Y^n\ar[r]^{d_Y^n}&Y^{n+1} }]$.  Without loss of generality, we can assume that $\widetilde{\rho}=\widetilde{\delta}$. By assumption, we have the following commutative diagram in ${\rm mod}\C$
$$\xymatrix{\C(-, X^n)\ar@{-->}[d]_{\C(-, f^n)}\ar[r]^{\C(-, d_X^{n})}&\C(-, X^{n+1})\ar@{-->}[d]^{\C(-, f^{n+1})}\ar[r]&\widetilde{\delta}\ar[r]\ar@{=}[d]&0\\
\C(-, Y^n)\ar[r]^{\C(-, d_Y^{n})}&\C(-, Y^{n+1})\ar[r]&\widetilde{\rho}\ar[r]&0}$$
Then there are dashed maps which make the diagram commutative by the projectivity of representable functors. By the Yoneda lemma and the dual of \cite[Proposition 3.6(1)]{HLN}, there is a morphism of distinguished $n$-exangles $f^{\mr}: \langle X^{\mr},\delta\rangle \rightarrow \langle Y^{\mr},\rho\rangle$. In particular,  $(f^0)_*\delta=(f^{n+1})^*\rho$. Since $\rho\in \mathbb{F}(X^{n+1}, X^0)$,  $(f^0)_*\delta=(f^{n+1})^*\rho\in \mathbb{F}(Y^{n+1}, X^0)$. Hence $\delta\in \mathbb{F}(Y^{n+1}, Y^0)$ by \cite[Lemma 3.14]{HLN}.
\end{proof}

\begin{theorem}\label{main47}
Let  $(\C, \mathbb{E}, \mathfrak{s})$ be a skeletally small $n$-exangulated category with weak-kernel. Then the map $\mathbb{F}\mapsto {\rm def}\mathbb{F}$ gives an isomorphism of the following posets, where the poset structures are given by inclusion.

{\rm (1)} The poset of closed subbifunctors of $\mathbb{E}$;

{\rm (2)} The poset of  Serre subcategories of ${\rm def}\mathbb{E}$.
\end{theorem}

\begin{proof} The proof is a modification of \cite[Theorem B]{Eno}.

{\bf A map ${\rm def}(-)$ from (1) to (2)}.

Let $\mathbb{F}$ be a closed subbifunctor of $\mathbb{E}$. Then $(\mathcal{C}, \mathbb{F}, \mathfrak{s}|_{\mathbb{F}})$ is $n$-exangulated by \cite[Proposition 3.16]{HLN}. So ${\rm def}\mathbb{F}$ is a Serre subcategory of ${\rm mod}\C$ by Proposition \ref{Pro1}. Since every $\mathfrak{s}|_{\mathbb{F}}$-deflation is an $\mathfrak{s}$-deflation, ${\rm def}\mathbb{F}$ is a subcategory of ${\rm def}\mathbb{E}$. Thus ${\rm def}\mathbb{F}$ is a Serre subcategory of an abelian category of ${\rm def}\mathbb{E}$, thus we obtain a map ${\rm def}(-)$ from (1) to (2).

{\bf A map $\mathbb{F}(-)$ from (2) to (1)}.

Let $\mathcal{S}$ be a Serre subcategory of ${\rm def}\mathbb{E}$, and we will construct an additive subbifunctor $\mathbb{F}(\mathcal{S})$ of $\mathbb{E}$. Since ${\rm def}\mathbb{E}$ is a Serre subcategory of ${\rm mod}\C$, it immediately follows that $\mathcal{S}$ is also a Serre subcategory of ${\rm mod}\C$. Let $A$ and $C$ be objects in $\C$ and $\delta\in \mathbb{E}(C, A)$. Then we denote by $\mathbb{F}(\mathcal{S})(C, A)$ a subset of $\mathbb{E}(C, A)$ consisting of $\delta$ such that $\widetilde{\delta}$ belongs to $\mathcal{S}$.

We claim that $\mathbb{F}(\mathcal{S})$ actually defines a closed subbifunctor of $\mathbb{E}$. It suffices to prove that $\mathbb{F}(\mathcal{S})$ is a subbifunctor of $\mathbb{E}$ since we can prove that $\mathbb{F}(\mathcal{S})$ is an additive closed subbifunctor of $\mathbb{E}$ by a similar argument
to \cite[Theorem B]{Eno}.

 Let $\delta\in \mathbb{F}(\mathcal{S})(X^{n+1}, X^0)$ and
$\mathfrak{s}(\delta)=[ \xymatrix{X^0\ar[r]^{d_X^0}&X^1\ar[r]^{d_X^1}&\cdots\ar[r]&X^{n-1}\ar[r]^{d_X^{n-1}}&X^n\ar[r]^{d_X^n}&X^{n+1}} ]$. First take any $f^{n+1}: Y^{n+1}\rightarrow X^{n+1}$, then we show that $(f^{n+1})^*\delta$ belongs to $\mathbb{F}(\mathcal{S})(Y^{n+1}, X^0)$, that is, $\widetilde{(f^{n+1})^*\delta}\in\mathcal{S}$. Let  $Y^{\mr}$ be an $\mathfrak{s}$-conflation which realizes $(f^{n+1})^*\delta\in \mathbb{E}(Y^{n+1}, X^0)$.
Then $(1_{X^0}, f^{n+1}): (f^{n+1})^*\delta\rightarrow \delta$ has a good lift $f^{\mr}: \langle Y^{\mr}, (f^{n+1})^*\delta\rangle\rightarrow \langle X^{\mr}, \delta\rangle$ which makes $\langle M_{f^{\mr}}, (d_Y^0)_*\delta\rangle$ a distinguished $n$-exangle by (EA2), that is, there is a morphism of distinguished $n$-exangles with following commutative diagram
$$\xymatrix{X^0\ar[r]^{d_{Y}^{0}}\ar@{=}[d]&Y^1\ar[r]\ar[d]_{f^{1}}&\cdots \ar[r]&Y^{n-1}\ar[d]_{f^{n-1}}\ar[r]^{d_{Y}^{n-1}}& Y^{n}\ar[d]_{f^{n}}\ar[r]^{d_{Y}^{n}}& Y^{n+1} \ar[d]_{f^{n+1}} \ar@{-->}[r]^{\qquad(f^{n+1})^*\delta}&\\
   X^0\ar[r]^{d_{X}^{0}}&X^1\ar[r]&\cdots \ar[r]&X^{n-1}\ar[r]^{d_X^{n-1}}& X^{n}\ar[r]^{d_{X}^{n}}& X^{n+1} \ar@{-->}[r]^{\delta}&}$$
such that
$$\xymatrix@=3.3em{Y^1\ar[r]^{\tiny\begin{pmatrix}-d_Y^1\\f^1\end{pmatrix}\ \ \ \ }&Y^2\oplus X^1\ar[r]&\cdots\ar[r]&Y^n\oplus X^{n-1}\ar[r]^{\tiny\begin{pmatrix}-d_Y^n&0\\f^n&d_X^{n-1}\end{pmatrix}}&Y^{n+1}\oplus X^n\ar[r]^{\tiny\qquad \begin{pmatrix}f^{n+1}&d_X^n\end{pmatrix}}&X^{n+1}\ar@{-->}[r]^{\qquad (d_Y^0)_*\delta}&}$$
 is a distinguished $n$-exangle. By the Yoneda lemma, we obtain the following commutative diagram in ${\rm mod}\C$
$$\xymatrix@=3em{\C(-, Y^{n-1})\ar[d]^{\C(-, f^{n-1})}\ar[r]^{\C(-, d_Y^{n-1})}&\C(-, Y^n)\ar[d]^{\C(-, f^n)}\ar[r]^{\C(-, d_Y^n)}&\C(-, Y^{n+1})\ar[d]^{\C(-, f^{n+1})}\ar[r]^{\qquad p}&\widetilde{(f^{n+1})^*\delta}\ar[r]\ar[d]^{\alpha}&0\\
\C(-, X^{n-1})\ar[r]^{\C(-, d_X^{n-1})}&\C(-, X^n)\ar[r]^{\C(-, d_X^n)}&\C(-, X^{n+1})\ar[r]^{\qquad  \pi}&\widetilde{\delta}\ar[r]&0}$$
with exact rows. Now we claim that $\alpha: \widetilde{(f^{n+1})^*\delta}\rightarrow \widetilde{\delta}$ is an injection. For any $M\in\C$, if $m\in \widetilde{(f^{n+1})^*\delta}(M)$ satisfies $\alpha_M(m)=0$, there exists $a\in\C(M, Y^{n+1})$ such that $m=p_M(a)$ since $p_M$ is epic. Hence  $\pi_M\C(M, f^{n+1})(a)=\alpha_M p_M(a)=0$. There is $b\in\C(M, X^n)$ such that $\C(M, d_X^n)(b)=\C(M, f^{n+1})(a)$, that is, $d_X^nb=f^{n+1}a$ which implies ${\footnotesize\begin{pmatrix}f^{n+1}& d_X^n\end{pmatrix}\begin{pmatrix}-a\\b\end{pmatrix}}=d_X^nb-f^{n+1}a=0$. Hence there is a morphism ${\footnotesize\begin{pmatrix}a'\\b'\end{pmatrix}}: M\rightarrow Y^n\oplus X^{n-1}$ which makes the following diagram commutative
$$\xymatrix@=4em{&M\ar@{-->}[dl]_{\tiny\begin{pmatrix}a'\\b'\end{pmatrix}}\ar[d]^{\tiny\begin{pmatrix}-a\\b\end{pmatrix}}&\\
Y^n\oplus X^{n-1}\ar[r]^{\tiny\begin{pmatrix}-d_Y^n&0\\f^n&d_X^{n-1}\end{pmatrix}}&Y^{n+1}\oplus X^n\ar[r]^{\tiny\qquad \begin{pmatrix}f^{n+1}&d_X^n\end{pmatrix}}&X^{n+1}.}$$
Hence ${\footnotesize\begin{pmatrix}-a\\b\end{pmatrix}=\begin{pmatrix}-d_Y^n&0\\f^n&d_X^{n-1}\end{pmatrix}\begin{pmatrix}a'\\b'\end{pmatrix}
=\begin{pmatrix}-d_Y^na'\\f^na'+d_X^{n-1}b'\end{pmatrix}}$, and $a=d_Y^na'=\C(M, d_Y^n)(a')$ which implies $m=p_M(a)=p_M\C(M, d_Y^n)(a')=0$. Therefore $\alpha: \widetilde{(f^{n+1})^*\delta}\rightarrow \widetilde{\delta}$ is an injection. Since $\delta\in\mathbb{F}(\mathcal{S})(X^{n+1}, X^0)$, we have $\widetilde{\delta}\in \mathcal{S}$. Note that $\mathcal{S}$ is a Serre subcategory of ${\rm mod}\C$, we have $\widetilde{(f^{n+1})^*\delta}\in\mathcal{S}$, this means $(f^{n+1})^*\delta$ belongs to $\mathbb{F}(\mathcal{S})(Y^{n+1}, X^0)$.

Let $\delta\in \mathbb{F}(\mathcal{S})(X^{n+1}, X^0)$ and  $f^0: X^0\rightarrow Y^0$. By a similar argument of the last part of (Step 1) in \cite[Theorem B]{Eno}, we can prove that $(f^0)_*\delta\in \mathbb{F}(\mathcal{S})(X^{n+1}, Y^0)$. Hence $\mathbb{F}(\mathcal{S})$ is a subbifunctor of $\mathbb{E}$.

Finally, one can easily check that ${\rm def}\mathbb{F}(\mathcal{S})=\mathcal{S}$ and $\mathbb{F}({\rm def}\mathbb{F})=\mathbb{F}$, and their proofs are similar to that of  \cite[Theorem B]{Eno}.
\end{proof}

Now we apply Theorem \ref{main47} to study $n$-proper classes of distinguished $n$-exangles in an $n$-exangulated category $(\C, \E, \mathfrak{s})$. The notion of $n$-proper class was introduced in \cite{HZZ}, we omit some details here, but the reader can find them in \cite{HZZ}.

Recall from \cite{HZZ} that a morphism $f^{\mr}\colon\Xd\to\langle Y^{\mr},\rho\rangle$ of distinguished $n$-exangles is called a \emph{weak isomorphism} if
{$f^{0}$ and $f^{n+1}$} are isomorphisms. { It should be noted that the notion of weak isomorphisms defined here is different from the one defined by Geiss, Keller and Oppermann in $(n+2)$-angulated categories or by Jasso  in $n$-exact categories (see \cite{GKO} and \cite{J})}.

The most important property of $n$-proper classes is following.

\begin{lemma}\label{thma} {\rm\cite[Theorem 4.5]{HZZ}} Let $\xi$ be a class of distinguished $n$-exangles in $n$-exangulated category $(\C,\E,\s)$ which is closed under weak isomorphism.
Set $\mathbb{E}_\xi:=\mathbb{E}|_\xi$, that is, $$\mathbb{E}_\xi(C, A)=\{\delta\in\mathbb{E}(C, A)~|~\delta~ \textrm{is realized as a distinguished n-exangle} \ {_{A}\Xd_{C}} \ \textrm{in} \ {\xi}\}$$ for any $A, C\in{\C}$, and $\mathfrak{s}_\xi:=\mathfrak{s}|_{\mathbb{E}_\xi}$. Then $\xi$ is an  $n$-proper class if and only if $(\C, \mathbb{E}_\xi, \mathfrak{s}_\xi)$ is an $n$-exangulated category.
\end{lemma}

Now we have the following result on closed subbifunctors and $n$-proper subclasses.
\begin{lemma}\label{lemma:4.8} Let $(\C, \E, \mathfrak{s})$ be a skeletally small $n$-exangulated category. Then the map $\Y\mapsto \E_{\Y}$ gives an isomorphism between two posets, where poset structures are given by inclusion.

{\rm (1)} The poset of $n$-proper subclasses which is contained in $\xi_{\mathfrak{s}}$, where $\xi_{\mathfrak{s}}$ is the class of $\mathfrak{s}$-distinguished $n$-exangles in $(\C, \E, \mathfrak{s})$.

{\rm (2)} The poset of closed subbifunctors of $\E$.
\end{lemma}
\begin{proof}Let $\xi_{\mathfrak{s}}$  be a class of $\mathfrak{s}$-distinguished $n$-exangles in $(\C, \E, \mathfrak{s})$. It is easy to see that $\xi_{\mathfrak{s}}$ is an $n$-proper class on $\C$. Assume that $\Y$ is an $n$-proper class on $\C$ with $\Y\subseteq \xi$, then $\E_{\Y}\subseteq \E$ and $(\C, \E_{\Y}, \mathfrak{s}_{\Y})$ is an $n$-exangulated category by Lemma \ref{thma}. Hence $\E_{\Y}$ is a closed subbifunctor of $\E$, thus we obtain a map $\Y\mapsto \E_{\Y}$ from (1) to (2). It is easy to show that this map is order-preserving, and includes an embedding of posets (1) into (2).

Next we claim that every closed subbifunctor $\mathbb{F}$ of $\E$ is of the form $\E_{\Y}$ for some $n$-proper class $\Y$ satisfying $\Y\subseteq \xi_{\mathfrak{s}}$.
Let $\mathbb{F}$ be a closed subbifunctor of $\E$. Then $(\C, \mathbb{F}, \mathfrak{s}|_{\mathbb{F}})$ is an $n$-exangulated category by Proposition \ref{Pro}. Suppose that $\xi_{\mathfrak{s}|_{\mathbb{F}}}$ is the class of $\mathfrak{s}|_{\mathbb{F}}$-distinguished $n$-exangles, then $\xi_{\mathfrak{s}|_{\mathbb{F}}}$
is closed under weak isomorphisms. Indeed, if $f^{\mr}\colon\Xd\to\langle Y^{\mr},\rho\rangle$ is a weak isomorphism of distinguished $n$-exangles with $\langle Y^{\mr},\rho\rangle\in \xi_{\mathfrak{s}|_{\mathbb{F}}}$, then $\rho\in \mathbb{F}(Y^{n+1}, Y^0)$. Since $f^{\mr}$ is a weak isomorphism,  $f^0$ and $f^{n+1}$ are isomorphisms and $(f^0)_*\delta=(f^{n+1})^*\rho$. Hence $(f^0)_*\delta\in \mathbb{F}(X^{n+1}, Y^0)$, and $\delta\in \mathbb{F}(X^{n+1}, X^0)$ as $f^{0}$ is an isomorphism. Therefore, $\xi_{\mathfrak{s}|_{\mathbb{F}}}$
is closed under weak isomorphisms. It follows from Lemma \ref{thma} that $\xi_{\mathfrak{s}|_{\mathbb{F}}}$ is an $n$-proper class on $\C$ with $\xi_{\mathfrak{s}|_{\mathbb{F}}}\subseteq \xi_{\mathfrak{s}}$. Hence $(\C, \E_{\xi_{\mathfrak{s}|_{\mathbb{F}}}}, \mathfrak{s}_{\xi_{\mathfrak{s}|_{\mathbb{F}}}})$ is an $n$-exangulated category by Lemma \ref{thma} again. It is easy to see that $\mathbb{F}=\E_{\xi_{\mathfrak{s}|_{\mathbb{F}}}}$, as desired.
\end{proof}

As an immediate consequence of Theorem \ref{main47} and Lemma \ref{lemma:4.8}, we have the following classification of $n$-proper classes of a given skeletally small $n$-extriangulated category.

\begin{corollary}\label{corollary:4.9} Let $(\C, \E, \mathfrak{s})$ be a skeletally small $n$-exangulated category. Then there exists an isomorphism between two posets, where poset structures are given by inclusion.

{\rm (1)} The poset of $n$-proper subclasses which is contained in $\xi_{\mathfrak{s}}$, where $\xi_{\mathfrak{s}}$ is the class of $\mathfrak{s}$-distinguished $n$-exangles in $(\C, \E, \mathfrak{s})$.

{\rm (2)} The poset of  Serre subcategories of ${\rm def}\mathbb{E}$.
\end{corollary}

Next we apply Theorem \ref{main47} to study $n$-exact substructures. First we recall basics on  an $n$-exact category and $n$-exangulated category.
An $n$-exact category consists of a pair $(\C, \X)$, where $\C$ is an additive category and $\X$ is a class of $n$-exact sequences in $\C$ satisfies some axioms, see \cite[Definition 4.2]{J}. An $n$-exact category $(\C, \X)$ can be regarded as $n$-exangulated category as follows: For $A, C\in\C$, let $\mathbb{E}_{\X}(C, A)$ be the subclass of  $\Lambda^{n+2}_{(A,C)}$ consisting of all $[X^{\mr}]$ such that $[X^{\mr}]\in\X$, where $\Lambda^{n+2}_{(A,C)}$  denote the class of all homotopy equivalence class of $n$-exact sequences in $\mathbf{C}^{n+2}_{(A,C)}$. Then $\mathbb{E}_{\X}(C, A)$ is an abelian group, hence there is a biadditive functor $\mathbb{E}_{\X}: \C^{\rm op}\times \C\rightarrow \Ab$ by \cite[Proposition 4.32]{HLN}. Define the realization $\mathfrak{s}_{\X}(\delta)$ of $\delta\in\mathbb{E}_{\X}(C, A)$ to be itself, then $(\C, \mathbb{E}_{\X}, \mathfrak{s}_{\X})$ is an $n$-exangulated category, see \cite[Proposition 4.36]{HLN}.

Now we have the following result on closed subbifunctors and $n$-exact structures.
\begin{lemma}\label{lemma:4.10} Let $(\C, \X)$ be a skeletally small $n$-exact category. Then the map $\Y\mapsto \E_{\Y}$ gives an isomorphism of posets between two posets, where poset structures are given by inclusion.

{\rm (1)} The poset of $n$-exact structures which is contained in $\X$.

{\rm (2)} The poset of closed subbifunctors of $\E_{\X}$.
\end{lemma}

\begin{proof} Let $(\C, \E_{\X}, \mathfrak{s}_{\X})$ be  an $n$-exangulated category coming from an $n$-exact category $(\C, \X)$ as above. Then every $\mathfrak{s}|_{\X}$-inflation (resp. $\mathfrak{s}|_{\X}$-deflation) is a monomorphism (resp. an epimorphism). It follows from \cite[Proposition 4.23]{HLN} that $(\C, \X)$ satisfies condition (EI), see \cite[Definition 4.21]{HLN}. Assume that $\Y$ is an $n$-exact structure on $\C$ with $\Y\subseteq \X$, then we have an $n$-exangulated category $(\C, \E_{\Y}, \mathfrak{s}_{\Y})$ corresponding to $n$-exact structure $\Y$.  Hence $\E_{\Y}$ is a closed subbifunctor of $\E_{\X}$ by Proposition \ref{Pro}, thus we obtain a map $\Y\mapsto \E_{\Y}$ from (1) to (2). It is easy to show that this map is order-preserving, and includes an embedding of posets (1) into (2).

Next we claim that every closed subbifunctor $\mathbb{F}$ of $\E_{\X}$ is of the form $\E_{\Y}$ for some $n$-exact structure $\Y$ on $\C$ satisfying $\Y\subseteq \X$.
 Assume that $\mathbb{F}$ is a closed subbifunctor of $\E_{\X}$, then $(\C, \mathbb{F}, \mathfrak{s}|_{\mathbb{F}})$ is an $n$-exangulated category by Proposition \ref{Pro}.  Let $\Y$ denote the class of all $\mathfrak{s}|_{\mathbb{F}}$-conflation. It is easy to see that $(\C, \Y)$ satisfies condition (EI) by \cite[Definition 4.21]{HLN}. Since every $\mathfrak{s}|_{\mathbb{F}}$-inflation (resp. $\mathfrak{s}|_{\mathbb{F}}$-deflation) is an $\mathfrak{s}|_{\X}$-inflation (resp. $\mathfrak{s}|_{\X}$-deflation), it is a monomorphism (resp. an epimorphism).  Therefore, $(\C, \Y)$ is an $n$-exact category by  \cite[Propositions 4.37 and 4.23]{HLN}. Let $(\C, \E_{\Y}, \mathfrak{s}_{\Y})$ be an $n$-exangulated category corresponding to $n$-exact structure $\Y$. It is easy to see that $\mathbb{F}=\E_{\Y}$, as desired.
\end{proof}

As an immediate consequence of Theorem \ref{main47} and Lemma \ref{lemma:4.10}, we have the following classification of $n$-exact substructures of a given skeletally small $n$-exact category.
\begin{corollary}\label{corollary:4.11} Let $(\C, \X)$ be a skeletally small $n$-exact category. Then there exists an isomorphism of posets between two posets, where poset structures are given by inclusion.

{\rm (1)} The poset of $n$-exact structures which is contained in $\X$.

{\rm (2)} The poset of Serre subcategories of ${\rm def}\mathbb{E}_{\X}$.
\end{corollary}

\textbf{Jiangsheng Hu}\\
School of Mathematics and Physics, Jiangsu University of Technology,
 Changzhou, Jiangsu 213001, P. R. China.\\
E-mail: \textsf{jiangshenghu@jsut.edu.cn}\\[1mm]
\textbf{Yajun Ma}\\
Department of Mathematics, Nanjing University, Nanjing, Jiangsu 210093, China.\\
E-mail: \textsf{13919042158@163.com}\\[1mm]
\textbf{Dongdong Zhang}\\
Department of Mathematics, Zhejiang Normal University,
 Jinhua, Zhejiang 321004, P. R. China.\\
E-mail: \textsf{zdd@zjnu.cn}\\[1mm]
\textbf{Panyue Zhou}\\
College of Mathematics, Hunan Institute of Science and Technology, Yueyang, Hunan 414006, P. R. China.\\
E-mail: \textsf{panyuezhou@163.com}

\end{document}